%% file: non_prm.tex
\documentclass[a4paper,12pt]{amsart}
\usepackage{amsfonts,amsmath,amsthm}
\usepackage{mathrsfs,graphicx}
\usepackage{enumerate,colortbl,color}
\usepackage{amssymb,layout,hhline}

\newtheorem{theorem}{Theorem}[section]

\newtheorem{proposition}[theorem]{Proposition}

\newtheorem{lemma}[theorem]{Lemma}

\newtheorem{corollary}[theorem]{Corollary}

\theoremstyle{definition}

\newtheorem{definition}[theorem]{Definition}

\newtheorem{remark}[theorem]{Remark}
\newtheorem{example}[theorem]{Example}
\newcommand{\Z}{{\mathbb Z}}

\newcommand{\N}{{\mathbb N}}

\newcommand\B{{\mathcal B}}

\renewcommand\L{{\mathcal L}}

\makeatletter
\renewcommand{\p@enumii}{}
\makeatother

\numberwithin{equation}{section}

\DeclareMathOperator{\Orb}{Orb}

\DeclareMathOperator{\Sp}{Sp}

\numberwithin{equation}{section}

\date{}

\begin{document}

\title[Substitutions of some primitive components]{Invariant 
measures for subshifts arising from substitutions of some primitive components}
\author[M. Hama]{M. Hama}
\address{Department of English Language and Literature, 
Bunkyo Gakuin College,1-19-1 Mukougaoka, Bunkyo-ku, Tokyo 113-8668, JAPAN.}
\email{mhama@ell.u-bunkyo.ac.jp}

\author[H. Yuasa]{H. Yuasa}
\address{17-23-203 Idanakano-cho, Nakahara-ku, Kawasaki Kanagawa 211-0034, JAPAN.}
\email{hisatoshi\_yuasa@ybb.ne.jp}
\urladdr{http://www.geocities.jp/hjrkk606/}



\maketitle

\input{body.tex}


\end{document}

%% file: body.tex
\begin{abstract}
The class of substitutions of some primitive components is introduced. 
A bilateral subshift arising from a substitution of some primitive components 
is decomposed into pairwise disjoint, locally compact, 
shift\nobreakdash-invariant sets, on each 
of which an invariant Radon measure is unique up to scaling. It is 
completely characterized in terms of eigenvalues of an incidence matrix when 
the unique invariant measure is finite.
\end{abstract}

\section{Introduction}

It was shown by \cite{Fo,DHS} that any ``aperiodic'', stationary, properly 
ordered Bratteli diagram gives rise to a Bratteli\nobreakdash-Vershik 
system conjugate to a bilateral subshift arising from an aperiodic, primitive 
substitution, and vice versa. In \cite{BKM}, this correspondence was extended by 
successfully removing the hypothesis of both simplicity for Bratteli 
diagrams and primitivity for substitutions. Actually, they showed that if $B$ is 
a stationary, ordered Bratteli diagram which admits an 
aperiodic Vershik map $\lambda_B$ acting on a perfect space $X_B$, then the 
Bratteli\nobreakdash-Vershik system $(X_B,\lambda_B)$ is conjugate to a 
subshift arising from an aperiodic substitution if and only if no restriction 
of $\lambda_B$ to a minimal set is conjugate to an odometer; conversely, given 
an aperiodic substitution with nesting property, there exists a stationary, 
ordered Bratteli diagram yielding a Bratteli\nobreakdash-Vershik system 
conjugate to a subshift arising from the substitution. In contrast to 
\cite{Fo,DHS}, it is remarkable that ordered Bratteli diagrams with more than 
one minimal or maximal paths play central roles in the 
above\nobreakdash-mentioned correspondence \cite{BKM}; see also \cite{M}. 

On the other hand, aiming at a similar result in the class of almost 
simple, ordered Bratteli diagrams in the sense of \cite{Danilenko}, the 
second author \cite{Y} introduced the notion of almost primitivity for 
substitutions, and showed that an almost primitive substitution generates an 
almost minimal subshift, in the sense of \cite{Danilenko}, with a unique (up to 
scaling), nonatomic, invariant Radon measure. Before this work, as pointed out 
in \cite{BKMS}, a concrete almost primitive substitution was studied in 
\cite{Fisher}, which is the so\nobreakdash-called Cantor substitution. By 
\cite{HPS,Danilenko}, any 
almost minimal system is conjugate to the Vershik map arising from an almost 
simple, ordered Bratteli diagram. It is still an open question to 
characterize a class of almost simple, ordered Bratteli diagrams whose 
Vershik maps conjugate to subshifts arising from almost primitive substitutions. 
Actually, this question forces us to be in a quite different situation from 
\cite{Fo,DHS,BKM}: there exists a class of {\em non\nobreakdash-stationary}, 
almost simple, ordered Bratteli diagrams whose Vershik maps are conjugate to 
subshifts arising from almost primitive substitutions; see for details 
\cite[Remark~5.5]{Y}.

Applying the correspondence \cite{BKM} mentioned above, and hence via 
stationary, ordered Bratteli diagrams, invariant measures for subshifts arising 
from aperiodic substitutions were studied in \cite{BKMS}. Roughly speaking, 
one of their results showed the existence of a one\nobreakdash-to\nobreakdash-one 
correspondence between the set of ergodic, 
probability (resp.\ nonatomic, infinite) measures for the subshift $X_\sigma$ 
arising from a given aperiodic substitution $\sigma$ and the set of 
``distinguished'' eigenvectors 
(resp.\ non\nobreakdash-distinguished eigenvalues) of the incidence matrix 
matrix $M_\sigma$ of $\sigma$. The goal of this paper is to refine 
this correspondence in the class of substitutions of {\em some 
primitive components} (Definition~\ref{n_components}). This class is so 
large that it includes all the primitive or almost primitive ones, and the 
so\nobreakdash-called Chacon substitution as well. 
Some properties required for a substitution to be of some primitive 
components are stronger, but the other is weaker, than those studied in 
\cite{BKMS}. We actually show that a bilateral subshift $X_\sigma$ arising from 
a given substitution $\sigma$ of some primitive components is decomposed into 
finite number of pairwise disjoint, locally compact, shift\nobreakdash-invariant 
sets $X_i$, $1 \le i \le N$, so that an invariant Radon measure on each 
$X_i$ is unique up to scaling, and moreover, the orbit of any point in each 
$X_i$ is dense in $X_i$. We also obtain the same criterion in terms of 
eigenvalues of $M_\sigma$ as \cite{BKMS} to determine when the unique 
invariant measure is finite. 

All the way to the end, we do not exploit Bratteli diagrams but tools within the 
framework of subshifts. This standing position is quite different from 
\cite{BKMS}. A characterization 
(Lemma~\ref{unique_Radon}) when a locally compact minimal subshift over a finite 
alphabet has a unique (up to scaling) invariant Radon measure will help us prove 
Theorem~\ref{uniqueness}. In Section~\ref{PFTheory}, auxiliary substitutions 
developed by \cite{Qu} play central roles when we estimate how fast the 
number of the occurrences of a letter in a $k$\nobreakdash-word of a given 
substitution of some primitive components increases as $k$ tends to infinity. 
The auxiliary substitutions also make it possible to 
calculate measures of cylinder sets with respect to invariant measures 
(Example~\ref{measures_of_cylinders}). 

\section{Substitutions in question}
We basically follow notation and terminology adopted in \cite{DHS} concerning 
combinatorics on words. 
Let $A$ be a finite alphabet with $\sharp A \ge 2$. Let 
$A^+$ denote the set of nonempty words over $A$. Set $A^\ast=A^+ \cup 
\{\Lambda\}$, where $\Lambda$ is the empty word. We say that 
$u \in A^+$ {\em occur} in $v \in A^+$, or $u$ is a {\em factor} of 
$v$, if there exists an integer $i$ with $1 \le i \le |v|$ such that 
$v_{[i,i+|u|)}:=v_iv_{i+1}\dots v_{i+|u|-1}=u$, where $|v|$ denotes the 
length of $v$ and $v_n$ is the $n$\nobreakdash-th letter of $v$. 
We refer to $i$ as an {\em occurrence} of $u$ in $v$. 
Given $u, v \in A^+$, we denote by $N(u,v)$ the number of the occurrences of 
$u$ in $v$. 

A {\em substitution} $\sigma$ on $A$ is a map from $A$ to $A^+$. By 
concatenations of words, we may define powers 
$\sigma^k:A \to A^+$ of $\sigma$ for $k \in \N$, and may enlarge the domain of 
the powers to $A^+$ or $A^\Z$. A subshift 
\[
X_\sigma=\{x=(x_i)_{i \in \Z} \in A^\Z; x_{[-i,i]}:=x_{-i}x_{-i+1} \dots x_i \in 
\L(\sigma) \textrm{ for every } i \in \N\}
\]
is called a {\em substitution dynamical system}, where 
\[
\L(\sigma)=\bigcup_{n \in \N, a \in A}\{w \in A^\ast; w 
\textrm{ is a factor of } \sigma^n(a)\}.
\] 
A word of the form $\sigma^n(a)$ is called an {\em $n$\nobreakdash-word}. 
We set $\L_n(\sigma)=\{w \in \L(\sigma);|w|=n\}$ for $n \in \N$. 
We denote by $T_\sigma$ the left shift on $X_\sigma$, and let 
$\Orb_{T_\sigma}(x)=\{{T_\sigma}^nx;n \in \Z\}$ for $x \in X_\sigma$. 
Given an infinite sequence $x$ over $A$, we set 
$\L(x)=\{w \in A^\ast;w \textrm{ is a factor of } x\}$, 
and $\L(X)=\bigcup_{x \in X}\L(x)$ if $X \subset A^\Z$. 
Given $u,v \in A^\ast$, we denote by $[u.v]$ a {\em cylinder set } 
$\{x \in X_\sigma;x_{[-|u|,|v|)}=uv\}$. 
If $u=\Lambda$, then we use the notation $[v]$ in stead of $[\Lambda.v]$. 
Given $x \in X_\sigma$, let $\delta_x$ denote the point mass concentrated on $x$. 
A positive measure on a locally compact metric space is called 
a {\em Radon measure} if it is finite on any compact set.

The {\em incidence matrix} $M_\sigma$ of $\sigma$ is an $A \times A$ matrix 
whose $(a,b)$\nobreakdash-entry is $N(b, \sigma(a))$. 
Putting a linear order on $A$, say $a_1<a_2<\dots<a_n$, we also write 
$(M_\sigma)_{i,j}$ to indicate the$(a_i,a_j)$\nobreakdash-entry 
$(M_\sigma)_{a_i,a_j}$. 
Let us recall some basic facts concerning square matrices. Let $M$ be a 
nonnegative square matrix. The matrix $M$ is said to be {\em primitive} if 
there exists $k \in \N$ such that $M^k>0$, i.e.\ every entry of 
$M^k$ is positive.
We denote by $\Sp(M)$ the set of eigenvalues of $M$. If $\lambda \in \Sp(M)$ is 
such that $|\eta| < \lambda$ for any other $\eta \in \Sp(M)$, then we call 
$\lambda$ a {\em dominant eigenvalue} of $M$. In this case,
\[
\min_i \sum_j M_{i,j} \le \lambda \le \max_i \sum_j M_{i,j}.
\]
Perron\nobreakdash-Frobenius 
Theory guarantees that any primitive matrix $M$ has a simple, dominant eigenvalue 
$\lambda$ which admits a positive eigenvector. Then, letting $\alpha$ and 
$\beta$ be positive, right and left eigenvectors of $M$ corresponding to 
$\lambda$, respectively, with $\beta \alpha=1$, it follows that 
$\lim_{k \to \infty}\lambda^{-k}(M^k)_{ij}=\alpha_i \beta_j$ for all possible 
$i,j$. See for details \cite{LM}. 

\begin{definition}\label{n_components}
A substitution $\sigma:A \to A^+$ is said to be {\em of some 
primitive components} if 
there is a sequence $\emptyset \ne A_1 \subsetneq A_2 \subsetneq 
\dots \subsetneq A_{n-1} \subsetneq A_n=A$ such that 
\begin{enumerate}[{\rm (i)}]
\item\label{closed}
for every integer $i$ with $1 \le i \le n$, it holds that 
$\sigma(a) \in {A_i}^+$ if $a \in A_i$;
\item
there exists $k \in \N$ such that for any integer $i$ with $1 \le i \le n$, 
any $a \in A_i \setminus A_{i-1}$ and any $b \in A_i$, the letter $b$ 
occurs in $\sigma^k(a)$, 
\end{enumerate}
where $A_0=\emptyset$. We also say that the substitution $\sigma$ is {\em of $n$ 
primitive components}. We call $n$ the {\em number of primitive 
components} of $\sigma$, and denote it by $n_\sigma$. 
\end{definition}

If a given substitution $\sigma$ is of some primitive components, then 
$M_\sigma$ is written in a form:
\begin{equation}\label{Q_i}
M_\sigma=
\begin{bmatrix}
Q_1 & 0 & 0 & \cdots & 0 \\
R_{2,1} & Q_2 & 0 & \cdots & 0 \\
R_{3,1} & R_{3,2} & Q_3 & \cdots & 0 \\
\vdots & \vdots & \vdots & \ddots & \vdots \\
R_{n_\sigma,1} & R_{n_\sigma,2} & R_{n_\sigma,3} & \cdots & Q_{n_\sigma}
\end{bmatrix}
\end{equation}
so that all the entries on or below diagonal of some power of ${M_\sigma}$ are 
positive. Conversely, this property implies that a given substitution is of 
some primitive components.

In the case $n_\sigma=1$, a substitution $\sigma$ becomes primitive. Then the 
subshift $X_\sigma$ is minimal and uniquely ergodic; see for example 
\cite{BFMS,Qu}. Throughout this paper, we assume 
$n_\sigma \ge 2$. Any almost primitive substitution is of two primitive 
components; see for details \cite{Y}. 

Definition~\ref{n_components} allows us to define a substitution 
$\sigma_i:A_i \to {A_i}^+$, $1 \le i \le n_\sigma$, by 
$\sigma_i(a)=\sigma(a)$ for $a \in A_i$. The substitution $\sigma_i$ is of 
$i$ primitive components. Since 
\[
\L(\sigma_1) \subset \L(\sigma_2) \subset \dots \subset \L(\sigma_{n_\sigma-1}) 
\subset \L(\sigma_{n_\sigma})=\L(\sigma),
\]
we have 
\[
X_{\sigma_1} \subset X_{\sigma_2} \subset \dots \subset X_{\sigma_{n_\sigma-1}} 
\subset X_{\sigma_{n_\sigma}}=X_{\sigma},
\]
which are all $T_\sigma$\nobreakdash-invariant closed sets. All of 
$X_{\sigma_2},X_{\sigma_3},\dots,X_\sigma$ are always 
nonempty. It holds that $X_{\sigma_1}=\emptyset$ if and only if $A_1$ is a 
singleton and $\sigma(s)=s$, where $A_1=\{s\}$. 

The class ${\mathcal S}$ of substitutions of some primitive components is 
different from the class ${\mathcal T}$ of substitutions studied in \cite{BKMS}. 
A substitution $\sigma:A \to A^+$ belongs to ${\mathcal T}$ if and only if 
the following conditions are satisfied:
\begin{enumerate}[{\rm (1)}]
\item\label{longer_and_longer}
$\lim_{n \to \infty}|\sigma^n(a)|=\infty$ for any $a \in A$;
\item\label{aperiodic}
$\sigma$ is aperiodic, that is, $X_\sigma$ has no periodic points of $T_\sigma$;
\item\label{admitted_form}
$M_\sigma$ is written in a form$:$ 
\begin{equation}\label{below_triangle}
M_\sigma=
\begin{bmatrix}
Q_1 & 0 & \cdots & 0 & 0 & \cdots & 0 \\
0 & Q_2 & \cdots & 0 & 0 & \cdots & 0 \\
\vdots & \vdots & \ddots & \vdots & \vdots &  & \vdots \\
0 & 0 & \cdots & Q_s & 0 & \cdots & 0 \\
R_{s+1,1} & R_{s+1,2} & \cdots & R_{s+1,s} & Q_{s+1} & \cdots & 0 \\
\vdots & \vdots &  & \vdots & \vdots & \ddots & \vdots \\
R_{m,1} & R_{m,2} & \cdots & R_{m,s} & R_{m,s+1} & \cdots & Q_m 
\end{bmatrix}
\end{equation}
so that 
\begin{enumerate}[{\rm (a)}]
\item
for every integer $i$ with $1 \le i \le m$, $Q_i$ is a primitive matrix if 
it is nonzero$;$
\item
for every integer $i$ with $s < i \le m$, there exists an integer $j$ with 
$1 \le j < i$ such that $R_{i,j} \ne 0$.
\end{enumerate}
\end{enumerate}
Notice that no inclusion relations hold between ${\mathcal S}$ and ${\mathcal T}$. 
Neither \eqref{longer_and_longer} nor \eqref{aperiodic} is not required for a 
substitution to be of some primitive components. However, the irreducible 
properties of incidence matrices required in \eqref{admitted_form} are not as 
rigid as those required for substitutions of some primitive components.

The following is a key lemma to investigate recurrence property and 
invariant sets for $X_\sigma$.
\begin{lemma}\label{basic_lemma}
Given an integer $i$ with $1 < i \le n_\sigma$, there exist $a \in A_{i-1}$, 
$b \in A_i \setminus A_{i-1}$, $k \in \N$, $u \in {A_{i-1}}^\ast$ and 
$v \in {A_i}^\ast$ such that at least one of the following holds$:$
\begin{enumerate}[{\rm (i)}]
\item\label{first_case}
$ab \in {\mathcal L}(\sigma_i)$ and $\sigma^k(ab)=uabv;$
\item\label{second_case}
$ba \in {\mathcal L}(\sigma_i)$ and $\sigma^k(ba)=vbau$.
\end{enumerate}
\end{lemma}
\begin{proof}
Put $r=\sharp A_i$. Find 
$a_0 \in A_{i-1}$ and $b_0 \in A_i \setminus A_{i-1}$ 
such that $a_0b_0 \in \L(\sigma_i)$ or $b_0a_0 \in \L(\sigma_i)$. It is enough to 
consider only the case $a_0b_0 \in \L(\sigma_i)$. 
Let $1 \le m_j < |\sigma^j(b_0)|$ be such that $\sigma^j(b_0)_{[1,m_j)} \in 
{A_{i-1}}^\ast$ and $b_j:=\sigma^j(b_0)_{m_j} \in A_i \setminus A_{i-1}$. 
Put $a_j = \sigma^j(a_0b_0)_{|\sigma^j(a_0)|+m_j-1}$. 
Since $a_{j_1}b_{j_1}=a_{j_2}b_{j_2}$ for some $j_2 > j_1 \ge 0$, 
\eqref{first_case} holds with $a=a_{j_1}$, $b=b_{j_1}$ and $k=j_2-j_1$. 
\end{proof}
We consider mainly the case 
where Lemma~\ref{basic_lemma}~\eqref{first_case} holds, 
because results below would be verified by means of symmetric arguments 
also for substitutions satisfying \eqref{second_case} of the lemma. 
Since $X_{\sigma^k}=X_\sigma$ for any $k \in \N$, we may assume 
Lemma~\ref{basic_lemma}~\eqref{first_case} with $k=1$.

\section{Recurrence property of $\sigma$}\label{periodic_points}

Throughout this section, we let $a,b,i,u$ and $v$ be as in 
Lemma~\ref{basic_lemma}~\eqref{first_case}. We are concerned with structure 
of $X_{\sigma_i} \setminus X_{\sigma_{i-1}}$. 
Consider the case $u = \Lambda$. Then $\sigma(a)=a$, and hence $\sigma(b)=bv$, 
$v \ne \Lambda$, $A_{i-1}=\{a\}$, $i=2$ and $X_{\sigma_1}=\emptyset$. 
\begin{lemma}[{\cite[Lemma~2.3]{Y}}]\label{arbitrarily_long}
Let $\tau:B \to B^+$ be a substitution such that $\tau(s)=s$ for some $s \in B$. 
Then the following are equivalent$:$
\begin{enumerate}[{\rm (i)}]
\item\label{arbitrary_power}
$s^p \in \L(\tau)$ for any $p \in \N;$
\item\label{prolonger}
there exist $c \in B \setminus \{s\}$, $k,l \in \N$ and $w \in B^\ast$ 
such that $\tau^k(c)=s^lcw$ or $\tau^k(c)=wcs^l$. 
\end{enumerate}
\end{lemma}
Assume Lemma~\ref{arbitrarily_long}~\eqref{prolonger} with $\tau=\sigma_2$. 
Then $B=A_2$ and $s=a$. If 
$w=\Lambda$, then $c=b$ and $X_{\sigma_2}=\{a^\infty\}$. 
If $w \ne \Lambda$, then $\sigma_2$ is almost 
primitive, so that $X_{\sigma_2}$ is almost minimal \cite[Theorem~3.8]{Y}. 

Assume the existence of $p_1 \in \N$ such that $a^p \notin \L(\sigma_2)$ if 
$p \ge p_1$. Then a letter of $A_2 \setminus A_1$ occurs in $v$, so that 
$b$ occurs infinitely many times in a fixed point 
$\omega^+:=bv\sigma(v)\sigma^2(v)\sigma^3(v)\dots \in {A_2}^\N$ of $\sigma$. This 
implies that there exists a periodic point $\omega \in X_{\sigma_2}$ of 
$\sigma$ such that $\omega_{[0,\infty)}=\omega^+$. It follows therefore 
that $\L(X_{\sigma_2})=\L(\omega)$ and hence 
$X_{\sigma_2}=\overline{\Orb_{T_\sigma}(\omega)}$. 

\begin{proposition}\label{Chacon_type}
$X_{\sigma_2}$ is minimal and uniquely ergodic.
\end{proposition}

\begin{proof}
Let $w \in \L(\omega)$. We may assume $\sigma(\omega)=\omega$. 
Take $k \in \N$ so that $w$ occurs in $\sigma^k(c)$ for 
any $c \in A_2 \setminus A_1$. Since 
$\omega=\dots\sigma^k(\omega_{-2})\sigma^k(\omega_{-1}).\sigma^k(\omega_0)
\sigma^k(\omega_1)\dots$, $w$ occurs in any factor  of $\omega$ whose length is 
$2\max_{c \in A_2 \setminus A_1}|\sigma^k(c)|+p_1$. This means the 
minimality of $X_{\sigma_2}$, since $w$ occurs infinitely often in $\omega$ 
with a bounded gap. 

The unique ergodicity will be proved below by using Theorem~\ref{uniqueness}.
\end{proof}

The Chacon substitution: $a \mapsto a$, $b \mapsto bbab$ satisfies the 
hypothesis of Proposition~\ref{Chacon_type}. The subshift $X_{\sigma_2}$ may 
be the orbit of a shift\nobreakdash-periodic point. If $A_2=\{a,b\}$ and 
$\sigma_2$ is defined by $a \mapsto a$ and $b \mapsto bab$, then $X_{\sigma_2}=
\{(ab)^\infty.(ab)^\infty, (ba)^\infty.(ba)^\infty\}$. It is a question whether 
$X_{\sigma_3}$ could be the orbit of a shift\nobreakdash-periodic point of 
periodic 3, or not.

We assume $u \ne \Lambda$ until Proposition~\ref{non_meager}. 
Consider the case $v = \Lambda$. 
Denoting $ua$ by $u$ afresh, we may write $\sigma(b)=ub$. Observe 
$A_i \setminus A_{i-1}=\{b\}$. In view of a configuration of words:
\begin{align*}
\sigma(b) & = & ub & \\
\sigma^2(b) & = & \sigma(u)ub & \\
\sigma^3(b) & = & \sigma^2(u)\sigma(u)ub & \\
\vdots & & \vdots & 
\end{align*}
we define a fixed point $\omega \in {A_i}^{-\N}$ of $\sigma$ by
$\omega =  \dots \sigma^5(u)\sigma^4(u)\sigma^3(u)\sigma^2(u)\sigma(u)ub$. 
Observe $\L(\omega)=\L(\sigma_i)$. 

Assume $u=a^{|u|}$. Then $A_{i-1}=\{a\}$, so that $i=2$. If $\sigma(a)=a$, then 
$X_{\sigma_2}=\{a^\infty\}$ and $X_{\sigma_1}=\emptyset$. If 
$\sigma(a)=a^p$ with $p \ge 2$, then $X_{\sigma_2}=X_{\sigma_1}=\{a^\infty\}$. 
Throughout the remainder of this section, we assume $u \ne a^{|u|}$. Then one 
of the following holds:
\begin{itemize}
\item
$i=2$ and $\sharp A_{i-1} \ge 2$;
\item
$i \ge 3$.
\end{itemize}

\begin{lemma}\label{inf_often}
Any word in $\L(X_{\sigma_i})$ occurs infinitely often in $\omega$. 
Consequently, $X_{\sigma_i} \subset {A_{i-1}}^\Z$.
\end{lemma}
\begin{proof}
Assume that some $v \in \L(X_{\sigma_i})$ occurs only finitely many times 
in $\omega$. Take $L \in \N$ so that $v$ does not occur in 
$\omega_{(-\infty,-L)}$. It follows that if $vw \in \L(X_{\sigma_i})$ then 
$|w| \le L$, which is a contradiction. 
\end{proof}

If $\{u_n\}_{n \in \N},\{v_n\}_{n \in \N} \subset A^+$ satisfy the 
properties: 
\begin{itemize}
\item
$u_n$ is a suffix of $u_{n+1}$ for each $n \in \N;$
\item
$v_n$ is a prefix of $v_{n+1}$ for each $n \in \N;$
\item
$\lim_{n \to \infty}|u_n|=\lim_{n \to \infty}|v_n|=\infty$,
\end{itemize}
we denote by $\lim_{n \to \infty}u_n.v_n$ a point 
$x \in A^\Z$ defined by $x_{[-|u_n|,|v_n|)}=u_nv_n$ for each $n \in \N$. 

\begin{proposition}\label{increasing_sequence}
Let $x \in X_{\sigma_i}$. Then the following are equivalent$:$
\begin{enumerate}[{\rm (i)}]
\item\label{nonempty}
$x \in X_{\sigma_i} \setminus X_{\sigma_{i-1}};$
\item\label{some_notin}
$x$ belongs to the orbit of a periodic point $y \in X_{\sigma_i} \setminus 
X_{\sigma_{i-1}}$ of $\sigma$, which is aperiodic under $T_\sigma$, such 
that for some $n \in \N$, $y_{[-n,n)} \notin \L(X_{\sigma_{i-1}})$ and 
$y_{(-\infty,-n-1]},y_{[n,\infty)} \in \L(X_{\sigma_{i-1}})$.
\end{enumerate}
The point $y$ would occur in one of the following fashions. 
If $\lim_{n \to \infty}|\sigma^n(c)|=\infty$ for any $c \in A$, then there exist 
$\gamma,\delta \in A_{i-1}$ and $q \in \N$ such that
\begin{itemize}
\item
$\gamma \delta \in \L(X_{\sigma_i}) \setminus \L(X_{\sigma_{i-1}});$ 
\item
$\sigma^{qj}(\gamma)_{|\sigma^{qj}(\gamma)|}=\gamma$ and 
$\sigma^{qj}(\delta)_1=\delta$ for any $j \in \N;$
\item
$y = \lim_{j \to \infty} \sigma^{qj}(\gamma).\sigma^{qj}(\delta)$.
\end{itemize}
If $A_1$ is a singleton, say $\{s\}$, then there exist $\gamma,\delta \in A_{i-1} 
\setminus A_1$, $p \in \N \cup \{0\}$ and $r \in \N$ such that 
\begin{itemize}
\item
$\sigma^{rj}(\delta)_{|\sigma^{rj}(\delta)|}=\delta$ and 
$\sigma^{rj}(\gamma)_1=\gamma$ for any $j \in \N;$
\item
$y$ is written in one of the following forms$:$
\[
y = \lim_{j \to \infty}s^j.\sigma^{rj}(\gamma), 
y = \lim_{j \to \infty}\sigma^{rj}(\delta).s^j \textrm{ and }
y = \lim_{j \to \infty}\sigma^{rj}(\delta).s^p\sigma^{rj}(\gamma).
\]
\end{itemize}
\end{proposition}
\begin{proof}
Assuming \eqref{nonempty}, we see \eqref{some_notin} in each of the cases:
\begin{enumerate}[(A)]
\item\label{CaseA}
$\lim_{n \to \infty}|\sigma^n(c)|=\infty$ for all $c \in A_1$;
\item\label{CaseB}
$A_1$ is a singleton, say $A_1=\{s\}$, and $\sigma(s)=s$.
\end{enumerate}
\paragraph{Case \eqref{CaseA}.} 
Take strictly increasing sequences $\{h_j\}_{j \in \N}, \{n_j\}_{j \in \N} 
\subset \N$ so that for each $j \in \N$, $x_{[-h_j,h_j)} \notin 
\L(X_{\sigma_{i-1}})$ and $\min_{c \in A_{i-1}}|\sigma^{n_j}(c)| \ge 2h_j$. 
Set
\[
K_j = \{k \in - \N ; x_{[-h_j,h_j)} \textrm{ occurs in } 
\sigma^{n_j}(\omega_{k-1}) \sigma^{n_j}(\omega_k)\}.
\]
Choose $c,d \in A_{i-1}$ so that there exist infinitely many $j$'s such that 
$\omega_{k_j-1}\omega_{k_j}=cd$ for some $k_j \in K_j$. 
By replacing $\{n_j\}_{j \in \N}$ with its appropriate subsequence, we may 
assume that for every $j \in \N$, $x_{[-h_j,h_j)}$ is a factor of 
$\sigma^{n_j}(c)\sigma^{n_j}(d)$, and $\sigma^{n_j}(c)_{|\sigma^{n_j}(c)|}$ and 
$\sigma^{n_j}(d)_1$ are 
constant, say $\gamma$ and $\delta$, respectively. We may assume furthermore 
that $n_{j+1}-n_j$ is constant, say $q$, which 
might be the least common multiple of $\min\{n \in \N;
\sigma^n(\gamma)_{|\sigma^n(\gamma)|}=\gamma\}$ and $\min\{n \in \N;\sigma^n(
\delta)_1=\delta\}$. There exists a sequence 
$\{w_j \in {A_{i-1}}^+;w_j \textrm{ is a factor of } 
\sigma^{qj}(\gamma \delta)\}_{j \in \N}$ 
which approximates $x$ arbitrarily close, because \eqref{CaseA} is assumed. 
Since $x_{[-h_j,h_j)} \notin \L(X_{\sigma_{i-1}})$, $\gamma \delta \notin 
\L(X_{\sigma_{i-1}})$ and each factor $x_{[-h_j,h_j)}$ of 
$\sigma^{qj}(\gamma)\sigma^{qj}(\delta)$ necessarily contains $\gamma \delta$ 
as a factor. Observe that $\gamma \delta$ occurs just once in $x_{[-h_j,h_j)}$, 
which would imply the shift\nobreakdash-aperiodicity of $x$. Then, 
${T_\sigma}^kx = \lim_{j \to \infty} \sigma^{qj}(\gamma).\sigma^{qj}(\delta)$ 
for some $k \in \Z$, which is a periodic point of $\sigma$. 

\paragraph{Case \eqref{CaseB}.} We have $i \ge 3$. 
By Lemma~\ref{arbitrarily_long}, 
$s^\infty \notin X_{\sigma_i} \setminus X_{\sigma_{i-1}}$. Take strictly 
increasing sequences $\{h_j\}_j, \{n_j\}_j \subset \N$ such that for every 
$j \in \N$, $x_{[-h_j,h_j)} \notin \L(X_{\sigma_{i-1}})$ and 
$\min_{c \in A_{i-1} \setminus A_1}|\sigma^{n_j}(c)| \ge 2h_j$. 
For each $j \in \N$, there exist $k_j \in -\N$ and $p_j,q_j \ge 0$ such that 
\begin{itemize}
\item
$x_{[-h_j,h_j)}$ occurs in $\sigma^{n_j}(\omega_{[k_j-q_j,k_j+p_j+1]})$;
\item
$\omega_{k_j}, \omega_{k_j+p_j+1} \in A_{i-1} \setminus A_1$;
\item
$\omega_{[k_j-q_j,k_j)}$ and $\omega_{[k_j+1,k_j+p_j]}$ are powers of $s$. 
\end{itemize}
Similar arguments to those in Case~\eqref{CaseA} may allow us to assume 
the existence of $w,w^\prime \in \L(\sigma_{i-1})$, $\delta,\gamma,\epsilon 
\in A_{i-1} \setminus A_1$ and $r \in \N$ such that 
\begin{itemize}
\item
for each $j \in \N$, $x_{[-h_j,h_j)}$ is a factor of 
$s^{q_j}\sigma^{rj}(w)s^{p_j}\sigma^{rj}(w^\prime)$;
\item
for each $j \ge 0$, the first and the last letters of $A_{i-1} \setminus A_1$ 
to occur in $\sigma^{rj}(w)$ are $\delta$ and $\gamma$, respectively;
\item
for each $j \ge 0$, the first letter of $A_{i-1} \setminus A_1$ to occur in 
$\sigma^{rj}(w^\prime)$ is $\epsilon$.
\end{itemize}
Let $m_j$ be an occurrence of $x_{[-h_j,h_j)}$ in 
$s^{q_j}\sigma^{rj}(w)s^{p_j}\sigma^{rj}(w^\prime)$. 
Put $m^\prime_j=m_j+2h_j-1$. It might be sufficient to consider the following 
cases:
\begin{enumerate}[(I)]
\item\label{I}
$\sharp \{j \in \N;1 \le m_j \le q_j\}=\infty$;
\item\label{II}
$\sharp \{j \in \N;0 < m_j - q_j \le |\sigma^{rj}(w)|\}=\infty$.
\end{enumerate}

Let us first consider Case~\eqref{I}. Let $l_j=\max\{l \ge 0;s^l 
\textrm{ is a prefix of } \sigma^{rj}(w^\prime)\}$. Assume 
$\sharp \{j \in \N;m^\prime_j \le q_j + |\sigma^{rj}(w)|+p_j+l_j\}=\infty$.  
Since 
$x \notin X_{\sigma_{i-1}} \cup  \{s^\infty\}$, there exists 
$\{j_k\}_k \subset \N$ such that $\lim_{k \to \infty}(q_{j_k}-m_{j_k})=
\infty$  or $\lim_{k \to \infty}(m^\prime_{j_k}-q_{j_k}-|\sigma^{rj_k}(w)|)=
\infty$. Hence, we may assume that either for every $k$, the first letter of 
$\sigma^{rj_k}(w)$ is $\gamma$, or for every $k$, the last letter of 
$\sigma^{rj_k}(w)$ is $\delta$. Then, one of the following holds:
\begin{itemize}
\item
${T_\sigma}^kx=\lim_{j \to \infty}s^j.\sigma^{rj}(\gamma)$ for some 
$k \in \Z$, and $s^p \gamma \notin \L(X_{\sigma_{i-1}})$ for some $p \in \N$;
\item
${T_\sigma}^kx=\lim_{j \to \infty}\sigma^{rj}(\delta).s^j$ for some $k \in \Z$, 
and $\delta s^p \notin \L(X_{\sigma_{i-1}})$ for some $p \in \N$.
\end{itemize}
This shows the conclusion. 

Let us assume 
$\sharp \{j \in \N; m^\prime_j > p_j+|\sigma^{rj}(w)|+q_j+l_j\}=\infty$. If 
for any $p \in \N$, $\delta s^p \epsilon$ occurs in $x$, then each of them 
occurs infinitely many times in $\omega$. However, it is 
impossible, because the last (resp.\ first) letter of $A_{i-1} \setminus A_1$ 
to occur in $\sigma^{rj}(\delta)$ (resp.\ $\sigma^{rj}(\epsilon)$) is $\delta$ 
(resp.\ $\epsilon$). Hence, $\{p_j\}_j$ is bounded, and the last (resp.\ first) 
letter of $\sigma^{rj}(\delta)$ (resp.\ $\sigma^{rj}(\epsilon)$) is $\delta$ 
(resp.\ $\epsilon$). Then, 
${T_\sigma}^k x=\lim_{j \to \infty}\sigma^{rj}(\gamma).\sigma^{rj}(s^p).
\sigma^{rj}(\delta)$ for some $k \in \Z$ and some $p \in \N$, and we have 
$\gamma s^p \delta \notin \L(X_{\sigma_{i-1}})$. The same argument works also 
in Case~\eqref{II}. This completes the proof.
\end{proof}

The following is an immediate consequence of 
Proposition~\ref{increasing_sequence}.
\begin{corollary}
There is a possibly empty set $\{x_j \in X_{\sigma_i} \setminus 
X_{\sigma_{i-1}};1 \le j \le N\}$ of periodic points of $\sigma$ such that 
\begin{enumerate}[{\rm (i)}]
\item
$\overline{\Orb_{T_\sigma}(x_j)}  =X_{\sigma_{i-1}} \cup \Orb_{T_\sigma}(x_j)$ 
$($a disjoint union$)$ for any integer $j$ with $1 \le j \le N;$
\item
$X_{\sigma_i} \setminus X_{\sigma_{i-1}} =\bigcup_{j=1}^N \Orb_{T_\sigma}(x_j)$ 
$($a disjoint union$)$.
\end{enumerate}
\end{corollary}

\begin{example}\label{ex1}
We shall see substitutions satisfying the hypothesis of 
Proposition~\ref{increasing_sequence}. 
\begin{enumerate}[(i)]
\item\label{unique_1}
Set $A=\{a,b,c\}$. Let $w \in \{a,b\}^+$. 
Define $\sigma: A \to A^+$ by $a \mapsto ab, b \mapsto a, c \mapsto wc$. 
Since $\{aa, ab, ba\} \subset \L(X_{\sigma_1})$ and $bb \notin 
\L(\sigma_2)$, we have $X_{\sigma_2} \setminus X_{\sigma_1} = \emptyset$. 
\item\label{unique_2}
Set $A=\{a,b,c,d\}$. Define $\sigma:A \to A^+$ by $a \mapsto abca, 
b \mapsto bacb, c \mapsto cbac, d \mapsto abbcad$.
Since $\{aa, bb\} \subset \L(\sigma_2) \setminus \L(X_{\sigma_1})$, 
\[
X_{\sigma_2} \setminus X_{\sigma_1} = 
\Orb_{T_\sigma}\left(\lim_{n \to \infty}\sigma^n(a).\sigma^n(a)\right) \cup 
\Orb_{T_\sigma}\left(\lim_{n \to \infty}\sigma^n(b).\sigma^n(b)\right).
\]
\item
Set $A=\{a,b,c,d,e\}$. Define $\sigma:A \to A^+$ by $a \mapsto a, b \mapsto 
cba, c \mapsto cbc, d \mapsto dc, e \mapsto bde$. It follows that 
$X_{\sigma_1}=\emptyset$, $X_{\sigma_2}$ is almost minimal with a unique 
fixed point $a^\infty$, 
\[
X_{\sigma_3} \setminus X_{\sigma_2}=
\Orb_{T_\sigma}\left(\lim_{n \to \infty}\sigma^n(c).\sigma^n(c)\right) 
\textrm{ and } 
X_{\sigma_4} \setminus X_{\sigma_3}=
\Orb_{T_\sigma}\left(\lim_{n \to \infty}a^n.\sigma^n(d)\right).
\]
\end{enumerate}
\end{example}

We assume $v \ne \Lambda$ until Proposition~\ref{non_meager}. 
In view of a configuration of words:
\begin{align*}
\sigma(ab) &=  & ua & bv \\
\sigma^2(ab) &= & \sigma(u)ua & b v \sigma(v) \\
\sigma^3(ab) &= & \sigma^2(u)\sigma(u)ua & b v \sigma(v) \sigma^2(v) \\
\vdots & & & \vdots
\end{align*}
we define a point $\omega \in X_{\sigma_i} \setminus X_{\sigma_{i-1}}$ by
\[
\omega = \dots \sigma^4(u)\sigma^3(u)\sigma^2(u)\sigma(u)ua. b v \sigma(v) 
\sigma^2(v) \sigma^3(v) \sigma^4(v) \dots
\]
Following \cite[Definition~2.5]{Y}, we make a definition:
\begin{definition}
We call the point $\omega$ a {\em quasi\nobreakdash-fixed point} of the 
substitution $\sigma$. We refer to a quasi\nobreakdash-fixed point of a 
power $\sigma^k$ as a {\em quasi\nobreakdash-periodic point} of $\sigma$. 
A quasi\nobreakdash-periodic point $x \in X_{\sigma_i} \setminus 
X_{\sigma_{i-1}}$ is said to be of a {\em primitive type} if it holds that 
$x_k \in A_i \setminus A_{i-1} \Leftrightarrow k=0$.
\end{definition}
It follows from the construction of $\omega$ that
\begin{enumerate}[(i)]
\item
for every $k \in \N$, there exists an integer $l_k$ with 
$0 \le l_k < |\sigma^k(b)|$ such that 
${T_\sigma}^{l_k}\sigma^k(\omega) = \omega$;
\item
$\L(\omega) = \L(\sigma_i)=\L(X_{\sigma_i})$;
\item
$\overline{\Orb_{T_\sigma}(\omega)}=X_{\sigma_i}$.
\end{enumerate}
Recall that $x \in A^\Z$ is said to be 
{\em positively recurrent} if for every $n \in \N$, there is $i \in \N$ such 
that $x_{[i-n,i+n)}=x_{[-n,n)}$.
\begin{lemma}\label{pos_rec}
\begin{enumerate}[{\rm (i)}]
\item
$\omega$ is aperiodic under ${T_\sigma}$. 
\item\label{positively_rec}
$\omega$ is positively recurrent if and only if $v \in {A_i}^+ 
\setminus {A_{i-1}}^+$.
\item\label{BG}
Put $K_n=\{k \in \Z; \omega_{[k,k+n)} \in {A_{i-1}}^+\}$ for $n \in \N$. For 
every $n \in \N$, there exists $m \in \N$ such that 
$\{l \le k < l+m;k \in K_n\} \ne \emptyset$ for any $l \in \Z$. 
\end{enumerate}
\end{lemma}
\begin{proof}
Since 
$\omega_{(-\infty, -1]} \in {A_{i-1}}^{-\N}$ and $\omega_0 \notin A_{i-1}$, 
$\omega$ is aperiodic under $T_\sigma$. 

If $v \in {A_{i-1}}^+$, then $\omega$ is not positively recurrent, because 
$\omega_n \ne \omega_0$ for every $n>0$. Suppose $v_j \in A_i \setminus 
A_{i-1}$ with $1 \le j \le |v|$. Take $k \in \N$ so that for any $c \in A_i 
\setminus A_{i-1}$, $ab$ occurs in $\sigma^k(c)$. Since for every $n \in \N$ and 
any $c \in A_i \setminus A_{i-1}$, $\sigma^n(ab)$ occurs in $\sigma^{n+k}(c)$, 
$\omega$ is positively recurrent. 

If $v \in {A_{i-1}}^+$, then \eqref{BG} is trivial. Assume $v \in 
{A_i}^+ \setminus {A_{i-1}}^+$. Given $n \in \N$, choose $p \in \N$ so that any 
word belonging to $\L_n(\sigma_i)$ occurs in $\sigma^p(c)$ for any 
$c \in A_i \setminus A_{i-1}$. Then, \eqref{BG} holds with 
$m=2\max\{\max_{c \in A_i \setminus A_{i-1}}|\sigma^p(c)|, n\}$. 
\end{proof}

If $u=a^{|u|}$, then $\sigma(a)=a^p$ for some $p \in \N$, which forces that 
$i=2$, $A_1=\{a\}$ and $\sigma_2$ is almost primitive. From now on, we assume 
$u \ne a^{|u|}$. 

Consider the case $v \in {A_{i-1}}^+$. It follows that 
$A_i \setminus A_{i-1}=\{b\}=\{\omega_0\}$, $\omega$ is of a primitive type, and 
$\overline{\Orb_{T_\sigma}(\omega)}=
X_{\sigma_{i-1}} \cup \Orb_{T_\sigma}(\omega)$ (a disjoint union).
\begin{proposition}\label{rec_trivial2}
There is a possibly empty set 
$\{x_j \in (X_{\sigma_i} \setminus X_{\sigma_{i-1}}) \cap {A_{i-1}}^\Z; 
1 \le j \le N\}$ of periodic points  of the substitution $\sigma$ such that 
\begin{enumerate}[{\rm (i)}]
\item
$X_{\sigma_i} \setminus X_{\sigma_{i-1}} = \Orb_{T_\sigma}(\omega) \cup 
\bigcup_{j=1}^N \Orb_{T_\sigma}(x_j)$ $($a disjoint union$);$
\item
if $x_j$ is periodic under $T_\sigma$, then 
$A_1$ is a singleton, say $\{s\}$, and $x_j=s^\infty;$
\item
if $x_j$ is aperiodic under $T_\sigma$, then 
$\overline{\Orb_{T_\sigma}(x_j)}=X_{\sigma_{i-1}} \cup \Orb_{T_\sigma}(x_j)$.
\end{enumerate}
\end{proposition}

\begin{proof}
Assume $x \in X_{\sigma_i} \setminus X_{\sigma_{i-1}}$. If $x_k \in A_i 
\setminus A_{i-1}$ for some $k \in \Z$, then $x = {T_\sigma}^k\omega$. 
If $x \in {A_{i-1}}^\Z$, then arguments in the proof of 
Proposition~\ref{increasing_sequence} work. 
\end{proof}

\begin{example}\label{ex2}
We shall see substitutions satisfying 
the hypothesis of Proposition~\ref{rec_trivial2}. 
\begin{enumerate}[(i)]
\item\label{unique_3}
Set $A=\{a,b,c,d\}$. Define $\sigma:A \to A^+$ by 
$a \mapsto abca$, $b \mapsto bacb$, $c \mapsto cbac$, $d \mapsto abadcac$. 
It follows that $\{aa,cc\} \subset \L(\sigma_2) \setminus \L(\sigma_1)$. 
Then,
\[
X_{\sigma_2} \setminus X_{\sigma_1} = \cup \Orb_{T_\sigma}(\omega) \cup 
\bigcup_{j=1}^2 \Orb_{T_\sigma}(x_j),
\]
where $x_1=\lim_{n \to \infty}\sigma^n(a).\sigma^n(a)$ and 
$x_2=\lim_{n \to \infty}\sigma^n(c).\sigma^n(c)$.
\item\label{unique_4}
Set $A=\{a,b,c,d\}$. Define $\sigma:A \to A^+$ by $a \mapsto ab, b \mapsto ab, 
c \mapsto acb, d \mapsto cdc$. Then, 
$X_{\sigma_1} = \Orb_{T_\sigma}(x)=\{x,T_\sigma x\}$, 
$X_{\sigma_2} \setminus X_{\sigma_1} = \Orb_{T_\sigma}(\omega)$ and 
$X_{\sigma_3} \setminus X_{\sigma_2} = \Orb_{T_\sigma}(\omega^\prime)$, 
where $x=(bc)^\infty.(bc)^\infty$, 
$\omega=\dots \sigma^2(a)\sigma(a)a .cb\sigma(b)\sigma^2(b)\dots$, and 
$\omega^\prime=\dots \sigma^2(c)\sigma(c)cdc\sigma(c)\sigma^2(c)\dots$. 
\item
Set $A=\{a,b,c,d,e\}$. Define $\sigma:A \to A^+$ by $a \mapsto a, b \mapsto 
cbab, c \mapsto cbc, d \mapsto adc, e \mapsto bdea$. 
Then, 
$X_{\sigma_1}=\emptyset$, $X_{\sigma_2}$ is minimal, 
\[
X_{\sigma_3} \setminus X_{\sigma_2}=\{a^\infty\} \cup 
\Orb_{T_\sigma}\left(\lim_{n \to \infty}\sigma^n(c).\sigma^n(c)\right) \cup 
\Orb_{T_\sigma}(\omega), 
\]
and $X_{\sigma_4} \setminus X_{\sigma_3}=\emptyset$, where 
$\omega=a^\infty.dc\sigma(c)\sigma^2(c)\sigma^3(c)\dots$. 
\end{enumerate}
\end{example}

We next consider the case $v \in {A_i}^+ \setminus {A_{i-1}}^+$.
\begin{definition}
\begin{enumerate}[(i)]
\item
Let $w \in {A_i}^+$. We refer to $w_{[m,n]} \in {A_{i-1}}^+$ as a 
{\em possible word} in $w$ if 
\[
w_{[m^\prime,n^\prime]} \in {A_{i-1}}^+, 1 \le m^\prime \le m, 
n \le n^\prime \le |u| \Rightarrow m^\prime = m, n^\prime = n.
\]
\item
Let $k^\prime \ge k \ge 1$ be integers and let $c \in A_i$. Suppose that 
$\sigma^k(c)_{[m,n]}$ (resp.\ $\sigma^{k^\prime}(c)_{[m^\prime, n^\prime]}$) 
is a possible word in $\sigma^k(c)$ (resp.\ $\sigma^{k^\prime}(c)$). 
We call $\sigma^k(c)_{[m,n]}$ an {\em ancestor} 
of $\sigma^{k^\prime}(c)_{[m^\prime, n^\prime]}$ if 
\[|\sigma^{k^\prime - k}(\sigma^k(c)_{[1,m)})| + 1 \ge m^\prime 
\textrm{ and } |\sigma^{k^\prime - k}(\sigma^k(c)_{[m,n]})| \le n^\prime.
\]
\end{enumerate}
\end{definition}

\begin{lemma}\label{sufficient_large_power}
Set 
$M = \max_{c, d \in A_i}\{|w|; w \textrm{ is a possible word in } \sigma(cd)\}$.
Let $p \in \N$. Suppose that $\sigma^k(c)_{[j,j+n)} \in {A_{i-1}}^+$ and 
$\sigma^k(c)_{j+n} \in A_i \setminus A_{i-1}$ for some 
$(c,k,j,n) \in A_i \setminus A_{i-1} \times 
\N \times \N \times \N$ with $1 \le j \le |\sigma^k(c)|$ and $n \ge (p + M)M$. Then, 
there exists $c^\prime \in A_i \setminus A_{i-1}$ such that
\[
\sigma^k(c)_{[j+n-n^\prime,j+n-n^\prime + |\sigma^p(c^\prime)|)}=\sigma^p(c^\prime),
\]
where $n^\prime = \max\{|w|;w \in {A_{i-1}}^+, w \textrm{ is a prefix of } 
\sigma^p(c^\prime)\}$. 
\end{lemma}

\begin{proof}
Let $j^\prime$ be such that $\sigma^k(c)_{[j^\prime,j+n)}$ is 
a possible word in $\sigma^k(c)$. There exists an integer $k^\prime$ with 
$1 \le k^\prime \le k$ such that $\sigma^k(c)_{[j^\prime,j+n)}$ does not have 
any ancestor in 
$\sigma^{k^\prime - 1}(c)$ but does in $\sigma^{k^\prime}(c)$. For each integer $l$ 
with $k^\prime \le l \le k$, let $\sigma^l(c)_{[j_l,j_l+n_l)}$ denote the ancestor of 
$\sigma^k(c)_{[j^\prime,j+n)}$. Since
\[
(p+M)M \le n + j - j^\prime \le n_{k^\prime}M + 
\sum_{l=k^\prime}^{k-1}(n_{l+1}-|\sigma(\sigma^l(c)_{[j_l,j_l+n_l)})|) 
\le M^2 + (k - k^\prime)M,
\]
we obtain $k-k^\prime \ge p$. The conclusion holds by taking $c^\prime$ to be 
the first letter of $A_i \setminus A_{i-1}$ to occur in 
$\sigma^{k - k^\prime -p}(\sigma^{k^\prime}(c)_{j_{k^\prime}+n_{k^\prime}})$.
\end{proof}

\begin{proposition}\label{non_meager}
There is a possibly empty set $\{x_j \in (X_{\sigma_i} \setminus 
X_{\sigma_{i-1}}) \cap {A_{i-1}}^\Z;1 \le j \le N\}$ of periodic points of 
$\sigma$ such that
\begin{enumerate}[{\rm (i)}]
\item\label{disjoint}
$\Orb_{T_\sigma}(x_j) \cap \Orb_{T_\sigma}(x_{j^\prime})=\emptyset$ if 
$j \ne j^\prime;$
\item
if $x_j$ is periodic under $T_\sigma$, then $A_1$ is a singleton, say 
$\{s\}$, and $x_j=s^\infty;$
\item\label{ape}
if $x_j$ is aperiodic under $T_\sigma$, then $\overline{\Orb_{T_\sigma}(x_j)}=
X_{\sigma_{i-1}} \cup \Orb_{T_\sigma}(x_j)$ $($a disjoint union$);$
\item
the orbit of any point in a $T_\sigma$\nobreakdash-invariant, locally compact 
$($non\nobreakdash-compact$)$ set$:$
\begin{equation}\label{X_i}
X_i:=X_{\sigma_i} \setminus \left(X_{\sigma_{i-1}} \cup 
\bigcup_{j=1}^N\Orb_{T_\sigma}(x_j)\right)
\end{equation}
is dense in $X_{\sigma_i}$, where we let $X_{\sigma_0}=\emptyset$.
\end{enumerate} 
\end{proposition}

\begin{proof}
Properties \eqref{disjoint}$\sim$\eqref{ape} are verified by the same argument 
as in the proof of Proposition~\ref{rec_trivial2}. 
Let $x^\prime \in X_i$. Let $w \in \L(X_{\sigma_i})$. Take $p \in \N$ 
so that $w$ is a factor of $\sigma^p(c)$ for all $c \in A_i \setminus A_{i-1}$. 
Fix $n \in \N$ with 
$n \ge \max\{ (p + M)M, |\sigma^p(c)|;c \in A_i \setminus A_{i-1}\}$,
where $M$ is as in Lemma~\ref{sufficient_large_power}. 
Lemma~\ref{pos_rec}~\eqref{BG} together with the fact that 
$\overline{\Orb_{T_\sigma}(\omega)}=X_{\sigma_i}$ enables us to find 
$l \in \Z$ such that $x^\prime_{[l,l+n)} \in {A_{i-1}}^+$ and 
$x^\prime_{l+n} \in A_i \setminus A_{i-1}$. Since $x^\prime_{[l,l+2n)}$ is a 
factor of $\sigma^k(c)$ for some $k \in \N$ and some 
$c \in A_i \setminus A_{i-1}$, 
Lemma~\ref{sufficient_large_power} ensures the existence of 
$d \in A_i \setminus A_{i-1}$ such that $\sigma^p(d)$ is a factor of 
$x^\prime_{[l,l+2n)}$. Hence $w$ is a factor of $x^\prime_{[l,l+2n)}$. This 
completes the proof.
\end{proof}

\begin{example}
The following substitutions satisfy the hypothesis of Proposition~\ref{non_meager}.
\begin{enumerate}[(i)]
\item
Set $A=\{a, b, c\}$. Define $\sigma:A \to A^+$ by $a \mapsto ab$, $b \mapsto a$, 
$c \mapsto acc$. Since $\sigma_1$ is primitive, $X_{\sigma_1}$ is minimal. The 
set $X_{\sigma_2} \setminus X_{\sigma_1}$ contains no periodic points of $\sigma$. 
\item
Set $A=\{a,b,c,d\}$. Define $\sigma:A \to A^+$ by $a \mapsto a$, $b \mapsto 
bbab$, $c \mapsto bcca$. Then, $X_{\sigma_1}=\emptyset$, $X_{\sigma_2}$ is 
minimal, and $X_{\sigma_3} \setminus X_{\sigma_2}$ contains $a^\infty$.
\end{enumerate}
\end{example}

Summarizing all the facts obtained above, we achieve the following.
\begin{theorem}\label{all_closed_inv_sets}
Let $\sigma:A \to A^+$ be a substitution of some primitive components. 
For each integer $i$ with $1 \le i \le n_\sigma$, we have a decomposition$:$
\[
X_{\sigma_i} \setminus X_{\sigma_{i-1}}=X_i \cup 
\Orb_{T_\sigma}(y_i)\cup \bigcup_{j=1}^{N_i}  
\Orb_{T_\sigma}(x_{ij})
\]
of $X_{\sigma_i} \setminus X_{\sigma_{i-1}}$ into possibly empty, locally 
compact, $T_\sigma$\nobreakdash-invariant sets 
$X_i$, $\Orb_{T_\sigma}(y_i)$ and $\Orb_{T_\sigma}(x_{ij})$ so that 
\begin{enumerate}[{\rm (i)}]
\item
$X_i$ is as in \eqref{X_i} if it is nonempty, and hence the orbit of any 
point in $X_i$ is dense in $X_{\sigma_i};$
\item
each $y_i$ is a quasi\nobreakdash-periodic point of a primitive type$;$
\item
each $x_{ij}$ is a periodic point of $\sigma$ such that if it is periodic under 
$T_\sigma$, then $A_1$ is a singleton, say $\{s\}$, and $x_{ij}=s^\infty;$ 
otherwise, $\overline{\Orb_{T_\sigma}(x_{ij})}=X_{\sigma_i} 
\cup \Orb_{T_\sigma}(x_{ij})$. 
\end{enumerate}
As a consequence, 
\begin{equation}\label{decomposition}
X_\sigma=\bigcup_{i=1}^{n_\sigma}X_i \cup
\bigcup_{i=2}^{n_\sigma}\Orb_{T_\sigma}(y_i) \cup 
\bigcup_{i=2}^{n_\sigma}\bigcup_{j=1}^{N_i}\Orb_{T_\sigma}(x_{ij}).
\end{equation}

The number of minimal sets of $X_\sigma$ is at most two. The two minimal sets 
are $X_{\sigma_2}$ and $\{s^\infty\}$, where $A_1=\{s\}$. 
The minimal set is unique if and only if one of the following holds$:$
\begin{enumerate}[{\rm (i)}]
\item\label{anyone_long}
$\lim_{n \to \infty}|\sigma^n(a)|=\infty$ for any $a \in A_1;$
\item\label{nonexistence}
$A_1$ is a singleton, say $\{s\}$, and $s^\infty \notin X_\sigma;$
\item\label{AP}
$A_1$ is a singleton and $\sigma_2$ is almost primitive.
\end{enumerate}
In these cases, the unique minimal set is $X_{\sigma_1}$, $X_{\sigma_2}$ and 
$\{s^\infty\}$, respectively, where $A_1=\{s\}$.

\end{theorem}

In Section~\ref{Vershik}, we will use the following characterization of 
shift\nobreakdash-periodic points in $X_\sigma$.

\section{Perron-Frobenius Theory for auxiliary substitutions}\label{PFTheory}

Let $\sigma:A \to A^+$ be a substitution of some primitive components. 
Let $i$ be an integer with $1 \le i \le n_\sigma$. Let $Q_1(i)$ be $Q_i$ in 
\eqref{Q_i}. Let $\theta_i$ denote a dominant eigenvalue of $Q_1(i)$. 
Given $m \in \N$, define a substitution 
$\sigma^{(m)}: \L_m(\sigma) \to \L_m(\sigma)^+$ by for $u \in \L_m(\sigma)$, 
\[
\sigma^{(m)}(u) = \sigma(u)_{[1,m]}, \sigma(u)_{[2,m+1]}, \sigma(u)_{[3,m+2]},
\dots , \sigma(u)_{[|\sigma(u_1)|,|\sigma(u_1)|+m-1]},
\]
where the commas between consecutive $\sigma(u)_{[i,m+i-1]}$'s are not new letters 
but just mean the separation between letters. Observe that there exists $k_0 \in 
\N$ such that any word in $\L(\sigma_i)$ occurs in $\sigma^k(a)$ for any $a \in 
A_i \setminus A_{i-1}$, any integer $i$ with $1 \le i \le n_\sigma$ and any 
integer $k \ge k_0$. 
Set ${\mathcal B}_m(i) = \{u \in \L_m(\sigma_{i+1});u_1 \in A_i\}$ 
for $0 \le i < n_\sigma$. Set $\lambda_i =\max_{1 \le j \le i} \theta_j$ and 
$\eta_i=\max_{i \le j \le n_\sigma}\theta_j$ for $1 \le i \le n_\sigma$, 
and $\lambda=\lambda_{n_\sigma}$.

We devote this section to analyzing how fast entries of ${M_{\sigma^{(m)}}}^k$ 
increase as $k$ tends to infinity. 
\begin{lemma}\label{start_for_PF}
With possibly empty matrices $F_{m,k}(i)$, $G_m(i)$, $Q_m(i)$ and $R_{m,k}(i)$, 
we may write that for every $k \in \N$,
\[
{M_{\sigma^{(m)}}}^k=\left[
\begin{array}{cccccccc}
{Q_{m}(1)}^k & 0 & 0 & 0 & 0 & \cdots & 0 & 0 \\
F_{m,k}(1) & {G_m(1)}^k & 0 & 0 & 0 & \cdots & 0 & 0 \\
\multicolumn{2}{c}{R_{m,k}(1)} & {Q_m(2)}^k & 0 & 0 & \cdots & 0 & 0 \\
\multicolumn{3}{c}{F_{m,k}(2)} & {G_m(2)}^k & 0 & \cdots & 0 & 0 \\
\multicolumn{4}{c}{R_{m,k}(2)} & {Q_m(3)}^k & \cdots & 0 & 0 \\
\multicolumn{5}{c}{\vdots} & \ddots & \vdots & \vdots \\
\multicolumn{6}{c}{F_{m,k}(n_\sigma-1)} & {G_m(n_\sigma-1)}^k & 0 \\
\multicolumn{7}{c}{R_{m,k}(n_\sigma-1)} & {Q_m(n_\sigma)}^k
\end{array}
\right],
\]
where 
\begin{itemize}
\item
$Q_m(i)$ is an $\L_m(\sigma_i) \setminus {\mathcal B}_m(i-1) \times 
\L_m(\sigma_i) \setminus {\mathcal B}_m(i-1)$ matrix$;$
\item
$G_m(i)$ is a 
${\mathcal B}_m(i) \setminus \L_m(\sigma_i) \times 
{\mathcal B}_m(i) \setminus \L_m(\sigma_i)$ matrix$;$
\item
$F_{m,k}(i)$ is a ${\mathcal B}_m(i) \setminus \L_m(\sigma_i) \times 
\L_m(\sigma_i)$ matrix$;$
\item
$R_{m,k}(i)$ is an $\L_m(\sigma_{i+1}) \setminus {\mathcal B}_m(i) \times 
{\mathcal B}_m(i)$ matrix.
\end{itemize}
Then, 
\begin{enumerate}[{\rm (i)}]
\item\label{when_1}
the following are equivalent$:$
\begin{enumerate}[{\rm (a)}]
\item
$\theta_i=1;$
\item
$Q_1(i)=\begin{bmatrix} 1 \end{bmatrix};$
\item
\[
X_{\sigma_i} \setminus X_{\sigma_{i-1}}=\Orb_{T_\sigma}(y) \cup \bigcup_{j=1}^N
\Orb_{T_\sigma}(x_j) \textrm{ $($a disjoint union$)$}
\]
for a quasi\nobreakdash-periodic point 
$y \in X_{\sigma_i} \setminus X_{\sigma_{i-1}}$ of a primitive type 
and for some periodic points $x_1,x_2,\dots,x_N \in X_{\sigma_i} 
\setminus X_{\sigma_{i-1}}$ of $\sigma$, some of which are possibly 
nonexistent$;$
\end{enumerate}
\item\label{when_empty}
the following are equivalent$:$
\begin{enumerate}[{\rm (a)}]
\item
$\L_m(\sigma_i) \setminus {\mathcal B}_m(i-1) = \emptyset;$
\item
$m > 1$, $A_i \setminus A_{i-1}$ is a singleton, say $\{s\}$, and 
$\sigma(s) = us$ for some $u \in {A_{i-1}}^+;$
\end{enumerate}
\item\label{Qsame}
$Q_m(i)$ is a primitive matrix with a dominant eigenvalue $\theta_i$, 
if $Q_m(i)$ is nonempty$;$
\item\label{1ika}
the absolute value of no eigenvalue of $G_m(i)$ is greater than one$;$
\item\label{below_bound}
there exists a set $\{c_i>0;1 \le i \le n_\sigma\}$ 
such that if $u \in \L_m(\sigma) \setminus \B_m(n_\sigma-1)$ and 
$v \in \L_m(\sigma_i) \setminus \L_m(\sigma_{i-1})$, then 
$({M_{\sigma^{(m)}}}^k)_{u,v} \ge c_i {\eta_i}^k$ for all sufficiently 
large $k \in \N$. 
\end{enumerate}
\end{lemma}
\begin{proof}
Assume that the size of $Q_1(i)$ is greater than one. Then, for a sufficiently 
large $k \in \N$,
\[
{\theta_i}^k={v_i}^{-1}\sum_j (Q_1(i)^k)_{ij}v_j>1, 
\]
where $v$ is a positive, right eigenvector corresponding to $\theta_i$. 
Hence, $\theta_i = 1$ implies $Q_1(i)=\begin{bmatrix} 1 
\end{bmatrix}$. 
The other parts in \eqref{when_1} is straightforward by using facts 
from Section~\ref{periodic_points}. Similarly, 
Statement~\eqref{when_empty} is readily verified. 

Assuming $Q_m(i)$ is nonempty, choose $k_0 \in \N$ so that any word in 
$\L_m(\sigma_i)$ occurs in $\sigma^k(a)$ for any integer $k \ge k_0$ and any 
$a \in 
A_i \setminus A_{i-1}$. This implies ${Q_m(i)}^k>0$ if $k \ge k_0$.
Define a $\L_m(\sigma_i) \setminus \B_m(i-1) \times A_i \setminus A_{i-1}$ 
matrix $N$ by $N_{u,a}=(M_\sigma)_{u_1,a}$ for each $(u,a)$. 
Then, $Q_m(i)N=NQ_1(i)$ because their $(u,a)$\nobreakdash-entries are 
$N(a, \sigma^2(u_1))$. If $Q_1(i)v=\theta_i v$ and $v>0$, then 
$Q_m(i)Nv=\theta_i Nv$ and $Nv>0$. This implies \eqref{Qsame}. 

If $({G_m(i)}^k)_{u,v}>0$, then $v \in 
\{\sigma^k(u)_{[j,j+m)};|\sigma^k(u_1)|-m+2 \le j \le |\sigma^k(u_1)|\}$. 
Hence, for any $k \in \N$, each row sum of ${G_m(i)}^k$ is not greater than 
$m-1$, which shows \eqref{1ika}. 

In the remainder of this proof, let us show \eqref{below_bound}. 
Take $k_0 \in \N$ so that $R_{m,k}(i)>0$ for any integer $k \ge k_0$ and 
any integer $i$ with $1 \le i < n_\sigma$. Put $M_i=M_{{\sigma_i}^{(m)}}$ for 
$1 \le i \le n_\sigma$. Put 
\[
M_i=
\begin{bmatrix}
M_{i-1} & 0 & 0 \\
F & G & 0 \\
R & R^\prime & Q
\end{bmatrix},
\]
where $G=G_m(i-1)$ and $Q=Q_m(i)$. Define $F^\prime, R_{k_0}, R_{k_0}^\prime$ 
by 
\[
{M_i}^{k_0}=\begin{bmatrix}
{M_{i-1}}^{k_0} & 0 & 0 \\
F^\prime & G^{k_0} & 0 \\
R_{k_0} & R_{k_0}^\prime & Q^{k_0}
\end{bmatrix}.
\]
Reducing ${M_i}^{k_0+k}={M_i}^{k_0}{M_i}^k={M_i}^k{M_i}^{k_0}$, we obtain that 
for every integer $k \in \N$,
\[
{M_i}^{k_0+k} \ge \begin{bmatrix}
{M_{i-1}}^{k_0+k} & 0 & 0 \\
0 & G^{k_0+k} & 0 \\
Q^kR_{k_0} & Q^kR_{k_0}^\prime & Q^{k_0+k}
\end{bmatrix}
\textrm{ and }
{M_i}^{k_0+k} \ge 
\begin{bmatrix}
{M_{i-1}}^{k_0+k} & 0 & 0 \\
0 & G^{k_0+k} & 0 \\
R_{k_0}{M_{i-1}}^k & {R_{k_0}}^\prime G & Q^{k_0+k}
\end{bmatrix}.
\]
This shows the conclusion in the case $i=n_\sigma$. Applying 
this argument to $M_{i-2}$ instead of $M_i$, we obtain the conclusion in the 
case $i=n_\sigma-1$. Repeating the argument, we may obtain \eqref{below_bound}.
\end{proof}

\begin{lemma}\label{eigen_vectors}
Set $i_{\min} =\min_{\theta_i=\lambda}i$ and 
$i_{\max} =\max_{\theta_i=\lambda}i$. Suppose $\lambda>1$. Then,
\begin{enumerate}[{\rm (i)}]
\item\label{r_e_vec}
a right eigenvector $\alpha=(\alpha_u)_{u \in \L_m(\sigma)}$ of 
$M_{\sigma^{(m)}}$ corresponding to $\lambda$ may be chosen so that 
$(\alpha_u)_{u \in {\mathcal B}_m(i_{\max}-1)}=0$ and 
$(\alpha_u)_{u \in \L_m(\sigma) \setminus {\mathcal B}_m(i_{\max}-1)}>0;$
\item\label{l_e_vec}
a left eigenvector $\beta=(\beta_u)_{u \in \L_m(\sigma)}$ of 
$M_{\sigma^{(m)}}$ corresponding to 
$\lambda$ may be chosen so that 
$(\beta_u)_{u \in \L_m(\sigma_{i_{\min}})}>0$ and 
$(\beta_u)_{u \in \L_m(\sigma) \setminus \L_m(\sigma_{i_{\min}})}=0;$
\item\label{geom_simple}
$\lambda$ is a simple, dominant eigenvalue of $M_{\sigma^{(m)}}$.
\end{enumerate}
\end{lemma}

\begin{proof}
Put $\alpha^\prime=(\alpha_u)_{u \in \B_m(n_\sigma-1)}$,  
$\alpha^{\prime \prime}=(\alpha_u)_{u \in \L_m(\sigma) \setminus 
\B_m(n_\sigma-1)}$ and 
\[
P_m(i) = 
\begin{bmatrix}
M_{{\sigma_i}^{(m)}} & 0 \\
F_{m,1}(i) & G_m(i)
\end{bmatrix}. 
\]
In order to prove \eqref{r_e_vec}, it is sufficient to show the following 
statements:
\begin{enumerate}[(a)]
\item
if $\theta_{n_\sigma} > \lambda_{n_\sigma-1}$, then $\alpha^\prime=0$ and 
$\alpha^{\prime \prime}$ may be chosen to be positive;
\item\label{xi}
if $\lambda_i=\lambda$ and $\xi$ is a right eigenvector of $P_m(i)$ 
corresponding to $\lambda_i$ such 
that $\xi^\prime:=(\xi_u)_{u \in \L_m(\sigma_i)} \ge 0$ and $\xi^\prime \ne 0$, 
then $(\xi_u)_{u \in \B_m(i) \setminus \L_m(\sigma_i)}>0$; 
\item\label{when_equal}
if $\theta_{n_\sigma}=\lambda_{n_\sigma-1}$ and $\alpha^\prime \ge 0$, then 
$\alpha^\prime=0$ and $\alpha^{\prime \prime}$ may be chosen to be positive;
\item\label{still_equal}
if $\theta_{n_\sigma} < \lambda_{n_\sigma-1}$, $\alpha^\prime \ge 0$ and 
$\alpha^\prime \ne 0$, then $\alpha^{\prime \prime}$ may be chosen to be 
positive.
\end{enumerate}
If $\theta_{n_\sigma} > \lambda_{n_\sigma-1}$, then clearly $\alpha^\prime=0$, and 
hence we may choose $\alpha^{\prime \prime}$ to be positive. Assuming the 
hypothesis of \eqref{xi}, since for a sufficiently large $k \in \N$,
\begin{equation}\label{contradiction1}
(\xi_u)_{u \in \B_m(i) \setminus \L_m(\sigma_i)}={\lambda_i}^{-k}
\sum_{j=0}^\infty\{{\lambda_i}^{-k}G_m(i)^k\}^jF_{m,k}(i)\xi^\prime>0,
\end{equation}
we obtain \eqref{xi}. Assume the hypothesis of \eqref{when_equal}. If 
$\alpha^\prime \ne 0$, then reducing ${M_{\sigma^{(m)}}}^k \alpha = 
{\theta_{n_\sigma}}^k \alpha$, we obtain a contradiction that for a sufficiently 
large $k \in \N$,
\begin{equation}\label{get_positive}
0 = \delta \{{\theta_{n_\sigma}}^kI-{Q_m(n_\sigma)}^k\} \alpha^{\prime 
\prime} = \delta R_{m,k}(n_\sigma-1)\alpha^\prime > 0,
\end{equation}
where $\delta$ is a positive, left eigenvector of $Q_m(n_\sigma)$ corresponding 
to $\theta_{n_\sigma}$. Statement~\eqref{still_equal} is obtained in the same 
manner as used to obtain \eqref{contradiction1}. 

Put
$\beta^\prime=(\beta_u)_{u \in \B_m(n_\sigma-1)}$ and 
$\beta^{\prime \prime}=(\beta_u)_{u \in \L_m(\sigma) \setminus \B_m(n_\sigma-1)}$.
In order to prove \eqref{l_e_vec}, it might be sufficient to show the following 
statements:
\begin{enumerate}[(1)]
\item
if $\lambda_{n_\sigma-1} \ge \theta_{n_\sigma}$, then 
$\beta^{\prime \prime}=0$;
\item\label{e}
if $\lambda_{n_\sigma-1} < \theta_{n_\sigma}$, then $\beta^\prime$ may 
be chosen to be positive;
\item\label{beta_P_m}
if $\lambda_{n_\sigma-1}=\lambda$ and $\xi$ is a left eigenvector of 
$P_m(n_\sigma-1)$ corresponding to $\lambda$, then 
$\xi_u=0$ for any $u \in \B_m(n_\sigma-1) \setminus \L_m(\sigma_{n_\sigma-1})$.
\end{enumerate}
Statement~\eqref{beta_P_m} follows from Lemma~\ref{start_for_PF}~\eqref{1ika}. 
If $\lambda_{n_\sigma-1}>\theta_{n_\sigma}$, then $\beta^{\prime \prime}=0$. 
Assume $\lambda_{n_\sigma-1}=\theta_{n_\sigma}$. If 
$\beta^{\prime \prime} \ne 0$, then we may assume $\beta^{\prime \prime}>0$. 
This yields a contradiction similar to \eqref{contradiction1}. 
Assume $\lambda_{n_\sigma-1}<\theta_{n_\sigma}$. We may assume 
$\beta^{\prime \prime}>0$. In a similar way to obtaining \eqref{get_positive}, 
we may see \eqref{e}. 
Statement~\eqref{geom_simple} is a consequence of the above argument. 
\end{proof}
\begin{remark}
Perron\nobreakdash-Frobenius Theory for matrices in \eqref{below_triangle} is 
discussed also in \cite{BKMS}. In fact, more general facts are stated in 
Theorem~3.1 therein. 
\end{remark}
\begin{lemma}\label{PF}
Suppose $\theta_i > 1$. If $\theta_i > \lambda_{i-1}$, then for any 
$u \in \L_m(\sigma_i) \setminus \B_m(i-1)$ and any $v \in \L_m(\sigma_i)$, 
\begin{equation}\label{conv}
\lim_{k \to \infty}\left({\theta_i}^{-k}{M_{\sigma^{(m)}}}^k\right)_{u,v}=
\alpha_u \beta_v >0.
\end{equation}

If $\theta_i \le \lambda_{i-1}$, then there exist 
$\{\gamma_u>0;u \in \L_m(\sigma_i) \setminus \B_m(i-1)\}$ and 
$\{\delta_v>0; v \in \L_m(\sigma_i) \setminus \L_m(\sigma_{i^\prime -1})\}$ 
such that 
\begin{equation}\label{div}
\lim_{k \to \infty}\left({\theta_i}^{-k}{M_{\sigma^{(m)}}}^k\right)_{u,v}=
\begin{cases}
\infty & \textrm{ if } v \in \L_m(\sigma_{i^\prime -1}); \\
\gamma_u \delta_v & \textrm{ otherwise}, 
\end{cases}
\end{equation}
where, putting 
$I=\{1 \le i_0 < i; \theta_{i_1} < \theta_i \textrm{ for any integer 
$i_1$ with }i_0 \le i_1 < i\}$, 
\[
i^\prime=
\begin{cases}
i & \textrm{ if }  I = \emptyset; \\
\min I & \textrm{ otherwise}.
\end{cases}
\]
\end{lemma}

\begin{proof}
We may assume $i=n_\sigma$. Let $s$ denote the size of $M_{\sigma^{(m)}}$.

Suppose $\theta_i > \lambda_{i-1}$. Let 
$N$ be a matrix which puts 
${\theta_i}^{-1}M_{\sigma^{(m)}}$ into a Jordan normal form:
\[
N^{-1}({\theta_i}^{-1}M_{\sigma^{(m)}})N=
\begin{bmatrix}
1 & 0 & 0 & \cdots & 0 \\
0 & \epsilon_1/\lambda & \ast & \cdots & 0 \\
0 & 0 & \epsilon_2/\lambda & \cdots & 0 \\
\vdots & \vdots & \vdots & \ddots & \vdots \\
0 & 0 & 0 & \cdots & \epsilon_{s-1}/\lambda
\end{bmatrix},
\]
where $\epsilon_1, \epsilon_2, \dots , \epsilon_{s-1}$ are eigenvalues of 
$M_{\sigma^{(m)}}$ other than $\lambda$. We obtain \eqref{conv}, since 
\[
N^{-1}\left(\lim_{k \to \infty} {\theta_i}^{-k}{M_{\sigma^{(m)}}}^k\right)N=
\begin{bmatrix}
1 & 0 & \cdots & 0 \\
0 & 0 & \cdots & 0 \\
\vdots & \vdots & \ddots & \vdots \\
0 & 0 & \cdots & 0
\end{bmatrix},
\]
and since the first column and the first row of $N$ and $N^{-1}$ are $\alpha$ 
and $\beta$, respectively. 

Assume $\theta_i = \lambda_{i-1}$. In this case, $\eta_j=\theta_i$ for any 
integer $j$ with $1 \le j < i$. Let $N$ be a matrix which puts 
${\theta_i}^{-1}M_{\sigma^{(m)}}$ into a Jordan normal form:
\[
N^{-1}({\theta_i}^{-1}M_{\sigma^{(m)}})N=J:=
\begin{bmatrix}
1 & 1/\lambda & \cdots & 0 & 0 & 0 & \dots & 0 \\
0 & 1 & \cdots & 0 & 0 & 0 & \dots & 0 \\
\vdots & \vdots & \ddots & \vdots & \vdots & \vdots & & \vdots \\
0 & 0 & \cdots & 1 & 1/\lambda & 0 &\cdots & 0 \\
0 & 0 & \cdots & 0 & 1 & 0 & \cdots & 0 \\
0 & 0 & \cdots & 0 & 0 & \epsilon_{r+1} / \lambda & \cdots & 0 \\
\vdots & \vdots & & \vdots & \vdots & \vdots & \ddots & \vdots \\
0 & 0 & \cdots & 0 & 0 & 0 & \cdots & \epsilon_s / \lambda
\end{bmatrix},
\]
where $\epsilon_{r+1},\epsilon_{r+2},\dots,\epsilon_s$ are eigenvalues of 
$M_{\sigma^{(m)}}$ other than $\lambda$. By Lemma~\ref{eigen_vectors}, 
the multiplicity $r$ of the eigenvalue $\lambda$ is greater than one. 

Set $\{1 \le p \le n_\sigma;\theta_p=\lambda\} = 
\{i_{\min}=i_1<i_2<\dots < i_r=n_\sigma\}$. Set $s_p =\sharp \L_m(\sigma_{i_p})$ 
for $1 \le p \le r$. 
Let $\xi$ be such that $\xi(M_{\sigma^{(m)}}-\lambda I)=\beta$. 
Put 
$\xi^\prime=(\xi_j)_{j=1}^{s_1}$ and 
$\xi^{\prime \prime}=(\xi_j)_{j > s_1}$. 
If $\xi^{\prime \prime}=0$, then 
$\xi^\prime(M_{\sigma_{i_1}^{(m)}}-\lambda I)=\beta^\prime:=
(\beta_j)_{j=1}^{s_1}$. This yields a contradiction that
$0=\xi^\prime(M_{\sigma_{i_1}^{(m)}}-\lambda I)\zeta=
\beta^\prime \zeta>0$, where $\zeta$ is a nonnegative, right eigenvector of 
$M_{{\sigma_{i_1}}^{(m)}}$ corresponding to $\lambda$. Hence, 
$\xi^{\prime \prime}$ is a left eigenvector of the matrix
\[
\begin{bmatrix}
G_m(i_1) & 0 & \cdots & 0 \\
\ast & Q_m(i_1+1) & \cdots & 0 \\
\vdots & \vdots & \ddots & \vdots \\
\ast & \ast & \cdots & Q_m(n_\sigma)
\end{bmatrix}
\] corresponding to $\lambda$. 
Using techniques developed in the proof of Lemma~\ref{eigen_vectors}, we may 
verify that given an 
integer $j$ with $s_1 < j \le s$, $\xi_j \ne 0$ if and only if 
$s_1 < j \le s_2$. 
Let $\rho$ be such that $\rho(M_{\sigma^{(m)}}-\lambda I)=\xi$. 
The same argument shows that given an integer $j$ with $s_2 < j \le s$, 
$\rho_j \ne 0$ if and only if $s_2 < j \le s_3$. 

Repeating this argument, we may see that given an integer $p$ with 
$1 \le p \le r$, the $p$\nobreakdash-th row $\xi$ of $N^{-1}$ satisfies the 
properties: 
\begin{itemize}
\item
$\xi_j \ne 0$ if $s_{r-p} < j \le s_{r-p+1}$;
\item
$\xi_j = 0$ if $s_{r-p+1} < j \le s$, 
\end{itemize}
where $s_0=0$. Since given integers $p,q$ with 
$1 \le p < r$ and $1 \le q \le r - p$, 
$\lim_{k \to \infty}(J^k)_{p,p+q}/k^t > 0$ if and only if $t=q$, it follows 
that under the extended arithmetics, 
\begin{align*}
& \lim_{k \to \infty}{\theta_i}^{-k}{M_{\sigma^{(m)}}}^k = 
N(\lim_{k \to \infty}J^k)N^{-1} \\
& =
\begin{bmatrix}
0 & & \\
\vdots & & \\
0 & & \\
\alpha_{r^\prime+1} & \parbox{12pt}{\Huge $\ast$}& \\
\alpha_{r^\prime+2} & & \\
\vdots & & \\
\alpha_s & &
\end{bmatrix}
\begin{bmatrix}
1 & \infty & \cdots & \infty & & \\
0 & 1 & \cdots & \infty & & \\
\vdots & \vdots & \ddots & \vdots & \parbox{12pt}{\Huge $0$} & \\
0 & 0 & \cdots & 1 & & \\
 & & & & & \\
& & \parbox{12pt}{\Huge $0$} && \parbox{12pt}{\Huge $0$} & \\
& & & & &
\end{bmatrix}
\begin{bmatrix}
\ast & \cdots & \ast & \ast & \cdots & \ast & \dots & 
\xi_{s_{r-1}+1} & \cdots & \xi_s \\
\vdots & & \vdots & \vdots & & \vdots & & \vdots &  & \vdots \\
\ast & \cdots & \ast & \xi_{s_2+1} & \cdots & \xi_{s_3} & \dots & 0 & \cdots & 0 \\
\beta_1 & \cdots & \beta_{s_1} & 0 & \cdots & 0 & \cdots & 0 & \cdots & 0 \\
 & & & & & & & &  & \\
& & & & \parbox{12pt}{\Huge $\ast$}& & & \\
& & & & & & & & & 
\end{bmatrix} \\
& =
\begin{bmatrix}
 & & & & & & \\
 & & \parbox{12pt}{\Huge $\ast$} & & & & \\
 & & & & & & \\
\alpha_{r^\prime+1} & \infty & \cdots & \infty & 0 & \cdots & 0 \\
\alpha_{r^\prime+2} & \infty & \cdots & \infty & 0 & \cdots & 0 \\
\vdots &  \vdots & & \vdots & \vdots & & \vdots \\
\alpha_s & \infty & \cdots & \infty & 0 & \cdots & 0 
\end{bmatrix}
\begin{bmatrix}
\ast & \cdots & \ast & \ast & \cdots & \ast & \dots & \xi_{s_{r-1}+1} & \cdots & 
\xi_s \\
\vdots & & \vdots & \vdots & & \vdots & & \vdots &  & \vdots \\
\ast & \cdots & \ast & \xi_{s_2+1} & \cdots & \xi_{s_3} & \dots & 0 & \cdots & 0 \\
\beta_1 & \cdots & \beta_{s_1} & 0 & \cdots & 0 & \cdots & 0 & \cdots & 0 \\
 & & & & & & & &  & \\
& & & & \parbox{12pt}{\Huge $\ast$}& & & \\
& & & & & & & & & 
\end{bmatrix} \\
& =
\begin{bmatrix}
& & & & & & \\
& & & & \parbox{12pt}{\Huge $\ast$} & & \\
& & & & & & \\
\infty & \infty & \cdots & \infty & \alpha_{r^\prime+1}\xi_{s_{r-1}+1} & 
\alpha_{r^\prime+1}\xi_{s_{r-1}+2} & \cdots & \alpha_{r^\prime+1}\xi_s \\
\infty & \infty & \cdots & \infty & \alpha_{r^\prime+2}\xi_{s_{r-1}+1} & 
\alpha_{r^\prime+2}\xi_{s_{r-1}+2} & \cdots & \alpha_{r^\prime+2}\xi_s \\
\vdots &  \vdots & & \vdots & \vdots & \vdots & & \vdots \\
\infty & \infty & \cdots & \infty & \alpha_s\xi_{s_{r-1}+1} & 
\alpha_s\xi_{s_{r-1}+2} & \cdots & \alpha_s\xi_s \\
\end{bmatrix},
\end{align*}
where $r^\prime = \sharp \B_m(n_\sigma-1)$. 

Assume $\theta_i < \lambda_{i-1}$. Put $M=\{(M_{\sigma^{(m)}})_{u,v}\}_{u,v 
\in \L_m(\sigma) \setminus \L_m(\sigma_{i^\prime -1})}$. 
It follows that $\theta_i$ is a simple, dominant eigenvalue of $M$. Since 
similar statements to 
Lemma~\ref{eigen_vectors} holds for this matrix $M$, arguments used to show 
\eqref{conv} of this lemma show the second half of \eqref{div}. 
The other half is obtained by using Lemma~\ref{start_for_PF}~\eqref{below_bound}. 
\end{proof}

\begin{example}\label{ex_PF}
\begin{enumerate}[(i)]
\item\label{ex_PF_1}
Set $A=\{a,b,c\}$. Define $\sigma:A \to A^+$ by $a \mapsto a^4, b \mapsto ab^3, 
c \mapsto cbc$. Then 
\[
M_\sigma=\begin{bmatrix}
4 & 0 & 0 \\
1 & 3 & 0 \\
0 & 1 & 2
\end{bmatrix}, \theta_1=4, \theta_2=3, \theta_3=2, \alpha=\begin{bmatrix}
2 \\ 2 \\ 1 \end{bmatrix}, \beta = \begin{bmatrix} 1 & 0 & 0 \end{bmatrix}. 
\]
It follows that $\L_2(\sigma)=\{aa,ab,ba,bb,bc,ca,cb\}$, $\L_2(\sigma_1)=\{aa\}$, 
$\B_2(1)=\{aa,ab\}$, $\L_2(\sigma_2)=\{aa,ab,ba,bb\}$, 
$\B_2(2)=\{aa,ab,ba,bb,bc\}$. We have
\begin{align*}
\sigma^{(2)} (aa) &= aa, aa, aa, aa; & \sigma^{(2)}(ab) &= aa, aa, aa, aa; & 
\sigma^{(2)}(ba) &= ab, bb, bb, ba; \\ 
\sigma^{(2)}(bb) &= ab, bb, bb, ba; & \sigma^{(2)}(bc) &= ab, bb, bb, bc; & 
\sigma^{(2)}(ca) &= cb, bc, ca; \\
\sigma^{(2)}(cb) &= cb, bc, ca,
\end{align*}
and
\[
M_{\sigma^{(2)}}=
\begin{bmatrix}
4 & 0 & 0 & 0 & 0 & 0 & 0 \\
4 & 0 & 0 & 0 & 0 & 0 & 0 \\
0 & 1 & 1 & 2 & 0 & 0 & 0 \\
0 & 1 & 1 & 2 & 0 & 0 & 0 \\
0 & 1 & 0 & 2 & 1 & 0 & 0 \\
0 & 0 & 0 & 0 & 1 & 1 & 1 \\
0 & 0 & 0 & 0 & 1 & 1 & 1
\end{bmatrix}, \alpha=
\begin{bmatrix}
2 \\ 2 \\ 2 \\ 2 \\ 2 \\ 1 \\ 1
\end{bmatrix}, \beta=
\begin{bmatrix}
1 & 0 & 0 & 0 & 0 & 0 & 0
\end{bmatrix}.
\]
\item
Set $A=\{a,b,c,d\}$. Define $\sigma:A \to A^+$ by $a \mapsto aa, b \mapsto 
ab^3c^3, c \mapsto abc^5, d \mapsto abcd^2$. Then
\[
M_\sigma=
\begin{bmatrix}
2 & 0 & 0 & 0 \\
1 & 3 & 3 & 0 \\
1 & 1 & 5 & 0 \\
1 & 1 & 1 & 2 
\end{bmatrix}, \theta_1 = 2, \theta_2 = 6, \theta_3 = 2, 
\alpha=\begin{bmatrix}
0 \\ 2 \\ 2 \\ 1
\end{bmatrix}, \beta=\begin{bmatrix}
1 & 1 & 3 & 0
\end{bmatrix}.
\]
It follows that $\L_2(\sigma)=\{aa,ab,bb,bc,ca,cc,cd,da,dd\}$, 
$\L_2(\sigma_1)=\{aa\}$, $\B_2(1)=\{aa,ab\}$, 
$\L_2(\sigma_2)=\{aa,ab,bb,bc,ca,cc\}$, $\B_2(2)=\{aa,ab,bb,bc,ca,cc,cd\}$. We 
have 
\begin{align*}
\sigma^{(2)}(aa) &= aa, aa; & \sigma^{(2)}(ab) &= aa, aa; \\
\sigma^{(2)}(bb) &= ab, bb, bb, bc, cc, cc, ca; &
\sigma^{(2)}(bc) &= ab, bb, bb, bc, cc, cc, ca; \\ 
\sigma^{(2)}(ca) &= ab, bc, cc, cc, cc, cc, ca; & 
\sigma^{(2)}(cc) &= ab, bc, cc, cc, cc, cc, ca; \\
\sigma^{(2)}(cd) &= ab, bc, cc, cc, cc, cc, ca; &
\sigma^{(2)}(da) &= ab, bc, cd, dd, da; \\ 
\sigma^{(2)}(dd) &= ab, bc, cd, dd, da, 
\end{align*}
and
\[
M_{\sigma^{(2)}}=
\begin{bmatrix}
2 & 0 & 0 & 0 & 0 & 0 & 0 & 0 & 0 \\
2 & 0 & 0 & 0 & 0 & 0 & 0 & 0 & 0 \\
0 & 1 & 2 & 1 & 1 & 2 & 0 & 0 & 0 \\
0 & 1 & 2 & 1 & 1 & 2 & 0 & 0 & 0 \\
0 & 1 & 0 & 1 & 1 & 4 & 0 & 0 & 0 \\
0 & 1 & 0 & 1 & 1 & 4 & 0 & 0 & 0 \\
0 & 1 & 0 & 1 & 1 & 4 & 0 & 0 & 0 \\
0 & 1 & 0 & 1 & 0 & 0 & 1 & 1 & 1 \\
0 & 1 & 0 & 1 & 0 & 0 & 1 & 1 & 1
\end{bmatrix}, 
\alpha = 
\begin{bmatrix}
0 \\ 0 \\ 2 \\ 2 \\ 2 \\ 2 \\ 2 \\ 1 \\ 1
\end{bmatrix}, 
\beta = \begin{bmatrix}
1 & 2 & 1 & 2 & 2 & 7 & 0 & 0 & 0
\end{bmatrix}.
\]
\end{enumerate}
\end{example}

\section{Invariant measures for $X_\sigma$}\label{inv_meas}

Let $\sigma:A \to A^+$ be a substitution of some primitive components. 
Let $i \in \N$ be $1 < i \le n_\sigma$. 
\begin{corollary}\label{freq}
Suppose $\theta_i > 1$. Let $a \in A_i \setminus A_{i-1}$, $v \in \L(\sigma_i)$, 
$m=|v|$ and $u \in \L_m(\sigma_i)$ with $u_1=a$. 
\begin{enumerate}[{\rm (i)}]
\item\label{prob}
If $\theta_i > \lambda_{i-1}$, then 
\begin{equation*}
\lim_{k \to \infty}\frac{N(v,\sigma^k(a))}{|\sigma^k(a)|}=\frac{\beta_v}{
\sum\limits_{w \in \L_m(\sigma_i)}\beta_w}.
\end{equation*}
\item\label{infinite}
If $\theta_i \le \lambda_{i-1}$, then  
\begin{equation*}
\lim_{k \to \infty}\frac{1}{{\theta_i}^k}N(v, \sigma^k(a))=
\begin{cases}
\infty & \textrm{ if } v \in \L_m(\sigma_{i^\prime-1}); \\
\gamma_u \delta_v & \textrm{ otherwise}. 
\end{cases}
\end{equation*}
\item
If $v \notin \L(\sigma_{i^\prime-1})$, then 
\begin{equation*}
\lim_{k \to \infty}\frac{N(v, \sigma^k(a))}
{\sum\limits_{b \in A_i \setminus A_{i-1}}N(b, \sigma^k(a))}=
\frac{\delta_v}{\sum\limits_{w \in \L_m(\sigma_i) \setminus \B_m(i-1)}
\delta_w}.
\end{equation*}
\end{enumerate}
\end{corollary}

\begin{proof}
Put $\tau=\sigma^{(m)}$. 
Assume $\theta_i > \lambda_{i-1}$. Since for every $k \in \N$, 
\begin{gather*}
|\sigma^k(a)| = |\tau^k(u)|; \notag \\
N(v, \tau^k(u))-(m-1) \le N(v, \sigma^k(a)) \le N(v,\tau^k(u)), \label{eqFreq}
\end{gather*}
it follows from Lemma~\ref{PF} that 
\[
\lim_{k \to \infty}\frac{N(v,\sigma^k(a))}{|\sigma^k(a)|}=\lim_{k \to \infty} 
\frac{N(v,\tau^k(u))}{|\tau^k(u)|}=\lim_{k \to \infty} 
\frac{({M_\tau}^k)_{u,v}}{\sum\limits_{w \in \L_m(\sigma_i)}
({M_\tau}^k)_{u,w}}=\frac{\beta_v}{\sum\limits_{w \in \L_m(\sigma_i)}\beta_w}.
\]
Assume $\theta_i \le \lambda_{i-1}$. Since
\[
\lim_{k \to \infty}\frac{1}{{\theta_i}^k}N(v, \sigma^k(a))=
\lim_{k \to \infty}\frac{1}{{\theta_i}^k}({M_\tau}^k)_{u,v},
\]
\eqref{infinite} follows from Lemma~\ref{PF}. 
It follows from Lemma~\ref{PF} again that
\[
\lim_{k \to \infty}\frac{N(v, \sigma^k(a))}
{\sum\limits_{b \in A_i \setminus A_{i-1}}N(b, \sigma^k(a))}=
\lim_{k \to \infty}
\frac{({M_\tau}^k)_{u,v}}
{\sum\limits_{w \in \L_m(\sigma_i) \setminus \B_m(i-1)}({M_\tau}^k)_{u,w}}
=\frac{\delta_v}{\sum\limits_{w \in \L_m(\sigma_i) \setminus \B_m(i-1)}
\delta_w}.
\]
\end{proof}
Suppose that $X_i$ in \eqref{decomposition} is nonempty. 
Lemma~\ref{start_for_PF}~\eqref{when_1} allows us to assume $\theta_i>1$. 
Let $\omega \in X_i$ be a quasi\nobreakdash-fixed point of the substitution 
$\sigma$, so that for each $k \in \N$, there are integers $m_k \ge 0$ and 
$n_k>0$ such that $\omega_{[-m_k,n_k)}=\sigma^k(\omega_0)$. Then, the 
following holds.
\begin{proposition}\label{existence}
\begin{enumerate}[{\rm (i)}]
\item\label{prob_meas}
If $\lambda_{i-1} < \theta_i$, then the weak$^\ast$ limit 
\[
\mu_i = 
\lim_{k \to \infty} \frac{1}{m_k+n_k}\sum_{j=-m_k}^{n_k-1}
\delta_{{T_\sigma}^j\omega}
\]
exists. 
\item\label{infinite_meas}
If $\lambda_{i-1} \ge \theta_i$, then $X_i$ has an infinite, invariant 
Radon measure $\nu_i$ characterized by the fact that for any $v \in \L(\sigma_i)$,
\[
\nu_i([v])=\lim_{k \to \infty}\frac{1}{{\theta_i}^k}\sum_{j=-m_k}^{n_k-1}
\delta_{{T_\sigma}^j\omega}([v]).
\]
\end{enumerate}
\end{proposition}

\begin{proof}
Let $v \in \L(\sigma_i)$. Put $m=|v|$. 

Assuming $\lambda_{i-1} < \theta_i$, it follows from 
Corollary~\ref{freq}~\eqref{prob} that 
\[
\mu_i([v])=\lim_{k \to \infty} \frac{1}{|\sigma^k(\omega_0)|}
N(v, \sigma^k(\omega_0))=\frac{\beta_v}{\sum\limits_{w \in \L_m(\sigma_i)}\beta_w}.
\]

Assume $\lambda_{i-1} \ge \theta_i$. Define an extended, 
real\nobreakdash-valued, set 
function $\tilde{\nu_i}$ on the ring ${\mathscr C}=\{[u.v];uv \in \L(\sigma_i)\}$
of cylinder sets by 
\[
\tilde{\nu_i}([u.v])=
\lim_{k \to \infty}\frac{1}{{\theta_i}^k}\sum_{j=-m_k}^{n_k-1}
\delta_{{T_\sigma}^j\omega}([u.v])=\lim_{k \to \infty}\frac{1}{{\theta_i}^k}
N(uv, \sigma^k(\omega_0))=\gamma_w \delta_{uv},
\]
where $w \in \L_{|uv|}(\sigma_i)$ with $w_1=\omega_0$. The set function 
$\tilde{\nu_i}$ is countably additive, and also, finite on any compact open 
subset of $X_i$. Hence, $\tilde{\nu_i}$ is uniquely extended to a 
$T_\sigma$\nobreakdash-invariant, Radon measure $\nu_i$ on $X_i$. 
It follows from Corollary~\ref{freq} that $\nu$ is infinite. 
\end{proof}

It might be worthwhile noticing that given an integer $j$ with $1 \le j \le i$, 
$\nu_i(X_{\sigma_j} \setminus X_{\sigma_{j-1}})=\infty$ iff $1 \le j < i^\prime$.

\begin{example}\label{measures_of_cylinders}
\begin{enumerate}[(i)]
\item
Let $\sigma$ be as in Example~\ref{ex_PF}~\eqref{ex_PF_1}. Then, $X_{\sigma_1}=\{
a^\infty\}$ has an invariant probability measure $\mu_1=\delta_{a^\infty}$, and 
$X_{\sigma_2} \setminus X_{\sigma_1}$ and $X_{\sigma_3} \setminus X_{\sigma_2}$ 
have infinite invariant measures $\nu_2$ and $\nu_3$, respectively. The measures 
of cylinder sets with respect to $\nu_2$ or $\nu_3$ can be calculated as follows:
\begin{align*}
\nu_2([a]) &=\infty, & \nu_2([b]) &= 1, & \nu_2([aa]) &= \infty, & 
\nu_2([ab]) &= 1/3, & \nu_2([ba]) &= 1/3, \\ 
\nu_2([bb]) &= 2/3, & \nu_3([a]) &= \infty, & \nu_3([b]) &= \infty, & 
\nu_3([c]) & = 1, & \nu_3([aa]) &= \infty, \\ 
\nu_3([aa]) &= \infty, & \nu_3([ab]) &= \infty, & \nu_3([ba]) &= \infty, 
& \nu_3([bb]) &= \infty, & \nu_3([bc]) &= 1, \\
\nu_3([ca]) &= 1/2, & \nu_3([cb]) &= 1/2.
\end{align*}
\item
Set $A=\{a,b,c,d,e\}$. Define $\sigma:A \to A^+$ by $a \mapsto ab, b \mapsto a, 
c \mapsto acd, d \mapsto adc, e \mapsto dece$. Then, $\theta_1=(1+\sqrt{5})/2$, 
$\theta_2=\theta_3=2$. It follows from Proposition~\ref{existence} that 
$X_{\sigma_1}$ and $X_{\sigma_2}$ have invariant probability measures $\mu_1$ 
and $\mu_2$, respectively, and that $X_{\sigma_3}$ has an infinite invariant 
measure $\nu_3$. The measures of cylinder sets with respect to $\mu_1$, 
$\mu_2$ and $\nu_3$ are calculated as follows:
\begin{align*}
\mu_1([a]) &= (\sqrt{5}-1)/2, & \mu_1([b]) &= (3-\sqrt{5})/2, & 
\mu_1([aa]) &= \sqrt{5}-2,\\ 
\mu_1([ab]) &=(3-\sqrt{5})/2, & \mu_1([ba])&=(3-\sqrt{5})/2, & 
\mu_2([a]) & = 1/2, \\ 
\mu_2([b]) &= 1/4, & \mu_2([c]) &= 1/8, & \mu_2([d]) &= 1/8 \\
\mu_2([aa]) &= 1/8, & \mu_2([ab]) &= 1/4, & \mu_2([ba]) &= 1/4, \\ 
\mu_2([ac]) &= 1/16, & \mu_2([ad]) &= 1/16, & \mu_2([ca]) &= 1/16, \\ 
\mu_2([cd]) &= 1/16, & \mu_2([da]) &= 1/16, & \mu_2([dc]) &= 1/16, \\
\nu_3([a]) &= \infty, & \nu_3([b]) &= \infty, & \nu_3([c]) &= \infty, \\
\nu_3([d]) &= \infty, & \nu_3([e]) &= 1, & \nu_3([aa]) &= \infty, \\
\nu_3([ab]) &= \infty, & \nu_3([ba]) &= \infty, & \nu_3([ac]) &= \infty, \\ 
\nu_3([ad]) &= \infty, & \nu_3([ca]) &= \infty, & \nu_3([cd]) &= \infty, \\ 
\nu_3([da]) &= \infty, & \nu_3([dc]) &= \infty, & \nu_3([dd]) &= 1/4, \\
\nu_3([ce]) &= 1/2, & \nu_3([de])&=1/2, & \nu_3([ea])&=1/2, \\ 
\nu_3([ec])&=1/2.
\end{align*}
\end{enumerate}
\end{example}

The following lemma plays a crucial role in showing that $\mu_i$ or $\nu_i$ 
is a {\it unique} invariant measure for $X_i$. 
\begin{lemma}\label{unique_Radon}
Let $X \subset A^\Z$ be a locally compact, minimal subshift. Let $T$ denote 
the left shift on $X$. Let $K \subset X$ be a 
nonempty, compact open set. Choose a point $\omega \in K$ which returns to 
$K$ infinitely many times, say at $0=k_0<k_1<k_2<\dots$. Then, the following 
are equivalent$:$
\begin{enumerate}[{\rm (i)}]
\item
$X$ has a unique $($up to scaling$)$, invariant Radon measure$;$
\item
for any $v \in \L(X)$ such that $[v] \subset K$, 
\[
\lim_{n \to \infty}\frac{1}{n}N(v, \omega_{[k_j,k_{j+n}]})
\]
converges to a constant uniformly in $j \ge 0$.
\end{enumerate}
Furthermore, if these conditions hold, then the unique invariant measure 
is ergodic.
\end{lemma}

\begin{proof}
If $X$ is non\nobreakdash-compact, then this is a consequence of 
\cite[Theorem~4.5]{Y}; see also the proof of \cite[Corollary~4.6]{Y}. 

If $X$ is compact, then it is sufficient to consider when the first return map 
$T_K$ induced on $K$ by $T$ is uniquely ergodic, because there exists a 
one\nobreakdash-to\nobreakdash-one correspondence between the set of 
$T$\nobreakdash-invariant probability measures and the set of 
$T_K$\nobreakdash-invariant probability measures. It follows from 
\cite[Theorem~IV.13]{Qu} that given a 
minimal homeomorphism $S$ on a totally disconnected, compact metric 
space $Y$, $S$ is uniquely ergodic if and only if for an arbitrarily chosen 
point $y \in Y$, it holds that for any nonempty, clopen set $F \subset Y$, 
there exists a constant $c$ such that
\[
\lim_{n \to \infty}\frac{1}{n}\sum_{i=0}^{n-1}\chi_F(S^{i+j}y) \to c 
\textrm{ as }n \to \infty, \textrm{ uniformly in }j \ge 0,
\]
where $\chi_F$ is the characteristic function of $F$. We obtain the conclusion 
by applying this criterion with $S=T_K$, $y=\omega$ and $F=[v]$. 
\end{proof}

\begin{theorem}\label{uniqueness}
Let 
\[
X_\sigma=\bigcup_{i=1}^{n_\sigma}X_i \cup
\bigcup_{i=2}^{n_\sigma}\Orb_{T_\sigma}(y_i) \cup 
\bigcup_{i=2}^{n_\sigma}\bigcup_{j=1}^{N_i}\Orb_{T_\sigma}(x_{ij})
\]
be the decomposition \eqref{decomposition}. 
The counting measure on each $\Orb_{T_\sigma}(x_{ij})$ is ergodic. The measure is 
finite if and only if the point $x_{ij}$ is a fixed point of $T_\sigma$. The 
counting measure on each $\Orb_{T_\sigma}(y_i)$ is an ergodic, infinite measure. 

If $X_i \ne \emptyset$, then an invariant Radon measure on $X_i$ provided with 
the relative topology is unique up to scaling, and ergodic. This measure is 
finite if $\theta_i > \lambda_{i-1}$, and infinite if $\theta_i \le 
\lambda_{i-1}$.
\end{theorem}

\begin{proof}
In view of Lemma~\ref{start_for_PF}~\eqref{when_1} and 
Proposition~\ref{existence}, it is enough for us to show the uniqueness of 
an invariant measure under the assumption $\theta_i>1$. Put 
$\{0=k_0<k_1<k_2<\dots\}=\{k \ge 0; \omega_k \in A_i \setminus A_{i-1}\}$. 
Suppose that a word $v \in \L(\sigma_i)$ contains a letter in 
$A_i \setminus A_{i-1}$ as a factor. Put
\[
\Delta_v=\frac{\delta_v}{\sum\limits_{w \in \L_m(\sigma_i) \setminus \B_m(i-1)}
\delta_w}.
\]
It is sufficient to prove that 
$N(v, \omega_{[k_j,k_{j+n}]})/n \to \Delta_v$ as $n \to \infty$, 
uniformly in $j \ge 0$. Let $\epsilon > 0$. Choose $p \in \N$ so that 
\begin{align*}
3\left\{ \min_{a,b \in A_i \setminus A_{i-1}}N(a, \sigma^p(b))\right\}^{-1} & < 
\frac{1}{4}|v|^{-1}\epsilon; \\
\left| N(v,\sigma^p(b)) - \Delta_v \sum_{a \in A_i \setminus A_{i-1}} 
N(a, \sigma^p(b)) \right| & < \frac{1}{4}\epsilon \sum_{a \in A_i \setminus 
A_{i-1}}N(a, \sigma^p(b)) \textrm{ for any } b \in A_i \setminus A_{i-1}.
\end{align*}
Choose an integer $n_0 \ge 2\max_{a \in A}|\sigma^p(a)|$ so that for all 
integers $n \ge n_0$,
\[
2\Delta_v n^{-1} \max_{a \in A}|\sigma^p(a)| + 
2 n^{-1} \max_{a \in A}|\sigma^p(a)| < \frac{1}{4}\epsilon.
\]
Since $\omega$ and $\sigma^p(\omega)$ coincide up to a shift by some digits, 
for every integer $j \ge 0$ 
there exist $q \in \N$, $r \in \N$ and $s,t \in A^\ast$ such that
\begin{itemize}
\item
$s$ is a suffix of $\sigma^p(\omega_{q-1})$;
\item
$t$ is a prefix of $\sigma^p(\omega_{q+r+1})$;
\item
$\omega_{[k_j,k_{j+n}]}=s\sigma^p(\omega_q)\sigma^p(\omega_{q+1}) \dots \sigma^p(
\omega_{q+r})t$.
\end{itemize}
Let $n \ge n_0$ and $j \ge 0$ be arbitrary integers. Since
\begin{multline*}
N(v,\omega_{[k_j,k_{j+n}]}) \le |s| + |t| + \\ 
\sum_{q \le l \le q+r \atop 
\omega_l \in A_i \setminus A_{i-1}}N(v,\sigma^p(\omega_l))
+ |v| \sharp \{
q-1 \le l \le q+r+1;\omega_l \in A_i \setminus A_{i-1}\}, 
\end{multline*}
we obtain
\begin{multline*}
\left| N(v,\omega_{[k_j,k_{j+n}]}) - \sum_{q \le l \le q+r \atop \omega_l \in 
A_i \setminus A_{i-1}} N(v, \sigma^p(\omega_l))\right| \le 2 \max_{a \in A}|
\sigma^p(a)| \\ 
+ |v| \sharp \{q-1 \le l \le q+r+1; \omega_l \in A_i \setminus A_{i-1}\}.
\end{multline*}
However, since 
\begin{align*}
& \frac{\sharp \{q-1 \le l \le q+r+1;\omega_l \in A_i \setminus A_{i-1}\}}
{\sum\limits_{q \le l \le q+r \atop \omega_l \in A_i \setminus A_{i-1}} 
\sum\limits_{a \in A_i \setminus A_{i-1}} N(a, \sigma^p(\omega_l))} \\
& \quad 
\le \frac{\sharp \{q-1 \le l \le q+r+1;\omega_l \in A_i \setminus A_{i-1}\}}
{\sharp \{q \le l \le q+r;\omega_l \in A_i \setminus A_{i-1}\} 
\min\limits_{a,b \in A_i \setminus A_{i-1}}N(a, \sigma^p(b))} \\
& \quad 
\le 3 \left\{ \min_{a,b \in A_i \setminus A_{i-1}}N(a, \sigma^p(b))\right\}^{-1} 
\le \frac{1}{4}|v|^{-1}\epsilon,
\end{align*}
we obtain
\begin{align*}
& \left| \frac{1}{n}N(v, \omega_{[k_j,k_{j+n}]}) - \frac{1}{n} 
\sum_{q \le l \le q+r \atop 
\omega_l \in A_i \setminus A_{i-1}}N(v, \sigma^p(\omega_l))\right| \\
& \quad < \frac{1}{4}\epsilon + |v| \cdot \frac{1}{4}|v|^{-1}\epsilon \cdot 
\frac{1}{n} \sum_{q \le l \le q+r \atop \omega_l \in A_i \setminus A_{i-1}} 
\sum_{a \in A_i \setminus A_{i-1}}N(a, \sigma^p(\omega_l)) \le 
\frac{1}{2}\epsilon.
\end{align*}
Also, we have
\begin{align*}
\left| \frac{1}{n}\sum_{q \le l \le q+r \atop \omega_l \in A_i \setminus A_{i-1}} 
N(v, \sigma^p(\omega_l)) - \Delta_v \right| & \le 
\frac{1}{n}\sum_{q \le l \le q+r \atop \omega_l \in A_i \setminus A_{i-1}} \left|
N(v, \sigma^p(\omega_l)) - \Delta_v \sum_{a \in A_i \setminus A_{i-1}} 
N(a, \sigma^p(\omega_l)) \right| \\
& \qquad + 2 \Delta_v n^{-1} \max_{a \in A}|\sigma^p(a)| \\
& < \frac{1}{4}\epsilon \cdot \frac{1}{n}\sum_{q \le l \le q+r \atop \omega_l 
\in A_i \setminus A_{i-1}} \sum_{a \in A_i \setminus A_{i-1}}N(a, 
\sigma^p(\omega_l)) + \frac{1}{4}\epsilon \le \frac{1}{2}\epsilon.
\end{align*}
Finally, 
\[
\left| \frac{1}{n}N(v, \omega_{[k_j,k_{j+n}]}) - \Delta_v \right| < \epsilon.
\]
This completes the proof.
\end{proof}

We are now in a position to see the unique ergodicity left to be 
proved in Proposition~\ref{Chacon_type}. Assume the hypothesis of 
the proposition, so that $\theta_1=1$ and $\theta_2>1$. Let 
$\omega$ be the fixed point of $\sigma$ as in the proposition. For this $\omega$, 
the proof of Theorem~\ref{uniqueness} may reach the same conclusion, that is, 
the first return map of $X_{\sigma_2}$ induced on $X_{\sigma_2} \setminus [a]$ 
is uniquely ergodic. This means the unique ergodicity of $X_{\sigma_2}$. 
\begin{corollary}\label{unique_ergodicity}
The subshift $X_\sigma$ is uniquely ergodic if and only if one of the 
following holds$:$
\begin{enumerate}[{\rm (i)}]
\item\label{X_sigma_1}
$\lambda=\theta_1>1;$
\item\label{X_sigma_2}
$\theta_1=1$, $\lambda=\theta_2>1$, and $s^\infty \notin X_\sigma$, where 
$A_1=\{s\}$.
\item\label{s^infty}
$\lambda=1;$
\end{enumerate}
\end{corollary}
\begin{proof}
Assume that $X_\sigma$ is uniquely ergodic. 
In view of Theorem~\ref{all_closed_inv_sets}, we first consider the case where 
$\lim_{n \to \infty}|\sigma^n(a)|=\infty$ for any $a \in A_1$. Since 
$X_{\sigma_1}$ is the unique minimal set and $\theta_1>1$, 
Theorem~\ref{uniqueness} implies that $\theta_i < \theta_1$ for any integer $i$ 
with $1 < i \le n_\sigma$. This corresponds to \eqref{X_sigma_1}. 
We then consider the case where $A_1$ is a singleton, say $\{s\}$, and 
$s^\infty \notin X_\sigma$. Then, $\sigma_2$ must satisfy the hypothesis of 
Proposition~\ref{Chacon_type}, and hence $\theta_2>1$. 
Theorem~\ref{uniqueness} implies $\theta_2>\theta_i$ for 
any integer $i$ with $2 < i \le n_\sigma$. This corresponds to 
\eqref{X_sigma_2}. We then consider 
the case where $A_1$ is a singleton and $\sigma_2$ is almost primitive. In this 
case, $\{s^\infty\}$ is the unique minimal set, where $A_1=\{s\}$. It follows 
from Theorem~\ref{uniqueness} again that 
$\theta_2=\theta_3= \dots = \theta_{n_\sigma}=1$. This corresponds to 
\eqref{s^infty}. The converse implication is straightforward in view of 
Lemma~\ref{start_for_PF}~\eqref{when_1}, Theorems~\ref{all_closed_inv_sets} 
and \ref{uniqueness}. 
\end{proof}

Among the examples studied above, uniquely ergodic systems are exactly 
\eqref{unique_1}, \eqref{unique_2} of Example~\ref{ex1}, 
\eqref{unique_3}, \eqref{unique_4} of Example~\ref{ex2}, and \eqref{ex_PF_1} 
of Example~\ref{ex_PF}. 

\begin{remark}
In \cite{BKMS}, $\theta_i$ is said to be distinguished if 
$\theta_i > \lambda_{i-1}$. It follows from 
Lemma~\ref{start_for_PF}~\eqref{when_1} and Theorem~\ref{uniqueness} that the 
case $\theta_i=1$ corresponds to the counting measure on $\Orb_{T_\sigma}(y_i)$ or 
$\Orb_{T_\sigma}(x_{ij})$. This kind of result is not obtained by \cite{BKMS}. 
Compare Theorem~\ref{uniqueness} and Corollary~\ref{unique_ergodicity} with 
\cite[Corollary~5.5]{BKMS}. 
\end{remark}